\documentclass{amsart}	
\usepackage{mathrsfs}
\usepackage{mathtools} 
\usepackage{amsthm}
\usepackage{amscd}
\usepackage{amssymb}
\usepackage{amsmath}
\usepackage{latexsym}
\usepackage{graphicx}
\usepackage{stmaryrd}
\usepackage{xypic}
\xyoption{curve}
\usepackage{nicefrac}
\usepackage{wasysym}
\usepackage{enumitem}
\usepackage{setspace}
\usepackage{ifthen}
\usepackage{verbatim}
\usepackage{eufrak}
\usepackage{relsize} 
\usepackage{tikz-cd}
\usepackage{stackengine} 
\usepackage{relsize} 
\usepackage[margin=1.5in]{geometry}
\usepackage[colorinlistoftodos,prependcaption,textsize=tiny]{todonotes}

\input{notation.tex} 

\newtheorem{thm}{Theorem}[section]
\newtheorem{cor}[thm]{Corollary}
\newtheorem{lem}[thm]{Lemma}

\newtheorem{prop}[thm]{Proposition}

\newtheorem*{lem*}{Lemma}
\newtheorem*{thm*}{Theorem}

\theoremstyle{definition}

\newtheorem{eg}[thm]{Example}
\newtheorem{rem}[thm]{Remark}

\newcommand{\dkcub}{\catfont{cub}}
\newcommand{\dkch}{\catfont{ch}}
\newcommand{\chains}{\catfont{ch}}

\newcommand{\PARSIMONIOUSBOX}{\boxplus}

\newcommand{\DISLATS}{\catfont{Dis}}
\newcommand{\CUBICALSETS}{\hat\BOX}
\newcommand{\STARS}{\catfont{Stars}}
\newcommand{\SIMPLICIALSETS}{\hat\DEL}
\newcommand{\pls}{\raisebox{.4\height}{\scalebox{.6}{+}}}
\newcommand{\mins}{\raisebox{.4\height}{\scalebox{.8}{-}}}

\newcommand{\dihomeo}{\varphi} 
\newcommand{\cnerve}{\catfont{ner}_\square}
\newcommand{\snerve}{\catfont{ner}_\DEL}

\newcommand{\cset}[1]{\ifthenelse{1=#1}{C}{\ifthenelse{2=#1}{A}{B}}}

\newcommand{\coface}{\delta}
\newcommand{\codeg}{\sigma}
\newcommand{\transpose}{\tau}
\newcommand{\coconnect}{\gamma}
\newcommand{\reversal}{\catfont{reverse}}
\newcommand{\diagonal}{\catfont{diag}}

\newcommand{\LEMBoxAutomorphisms}{Lemma 1, \cite{krishnan2023cubicalII}}
\newcommand{\LEMBoxEpimorphisms}{Lemma 2.9, \cite{krishnan2023cubicalII}}
\newcommand{\PROPUnderlyingMonoidalPreorderedSets}{Proposition 5.11, \cite{krishnan2009convenient}}
\newcommand{\ThmDiEmbed}{Theorem 2.5, \cite{fernandes2007classification}}

\newcommand{\PropTop}{Proposition 5.8, \cite{krishnan2009convenient}}

\newcommand{\PROPInclusions}{Proposition 5.6, \cite{krishnan2015cubical}}
\newcommand{\PropSD}{Proposition 7.4, 7.5, \cite{krishnan2015cubical}}
\newcommand{\THMPospaces}{Theorem 4.7, \cite{krishnan2009convenient}}
\newcommand{\PROPLocCompactStreams}{Theorem 5.4, \cite{krishnan2009convenient}}
\newcommand{\ThmDHomotopy}{Theorem 7.1, \cite{krishnan2015cubical}}
\newcommand{\THMVanKampen}{Theorem 1, \cite{krishnan2015cubical}}

\newcommand{\THMXClosed}{Theorem 5.12, \cite{krishnan2009convenient}}
\newcommand{\THMQCubical}{Theorem 4.18, \cite{jardine2006categorical}}
\newcommand{\THMTriEquivalence}{Theorem, \cite{jardine2002cubical}}
\newcommand{\CORTriEquivalence}{Theorem, \cite{jardine2002cubical}}
\newcommand{\THMRealizeEquivalence}{Theorem, \cite{jardine2002cubical}}
\newcommand{\THMAMS}{Theorem, \cite{jardine2002cubical}}
\newcommand{\THMTransfer}{Theorem 3.10, \cite{riehl2011algebraic}}
\newcommand{\THMPointwise}{Theorem 4.3, \cite{riehl2011algebraic}}
\newcommand{\THMClassicalEquivalence}{Theorem 4.3, \cite{riehl2011algebraic}}

\newcommand{\LEMQTCocontinuous}{Lemma 7.2, \cite{krishnan2015cubical}}
\newcommand{\LEMSdTri}{Lemma 7.2, \cite{krishnan2015cubical}}
\newcommand{\LEMMincubicalMonics}{Lemma 6.2, \cite{krishnan2015cubical}}
\newcommand{\EGMSpaces}{Examples 2.2 and 3.8, \cite{cole2006mixing}}
\newcommand{\SimplicialGraphs}{Lemma \ldots, \cite{krishnan2015cubical}}
\newcommand{\THMBirkhoff}{Theorem, {\cite{birkhoff1937rings}}}
\newcommand{\THMCat}{Proposition in \S4.2C, {\cite{gromov1987hyperbolic}}}
\newcommand{\LEMBoxInclusions}{Lemma 6.2, \cite{krishnan2015cubical}}
\newcommand{\COREZCellularModel}{Corollary 4.15, \cite{isaacson2011symmetric}}
\newcommand{\PROPTransferredRegularity}{Proposition 6.12, \cite{isaacson2011symmetric}}
\newcommand{\PROPTestFunctor}{Proposition 1.2.9, \cite{maltsiniotis2005theorie}}
\newcommand{\THMOldQuillenEquivalence}{Theorem 8.4.38, \cite{cisinskiprefaisceaux}}
\newcommand{\THMFrobenius}{Theorem 4.8, \cite{gambino2017frobenius}}
\newcommand{\THMCATZero}{Theorem 2.5, \cite{ardila2012geodesics}}
\newcommand{\PROPSemilatticeBirkhoff}{Proposoitiions 5.7,5.8, \cite{gonzalez2021finite}}
\newcommand{\THMBirkhoffDuality}{Proposoitiion 1, \cite{wraith1993using}}
\newcommand{\PROPGrothendieckTestCriterion}{Proposition, p.86 44(d), \cite{grothendieck1983pursuing}}
\newcommand{\THMMinimalSymmetricMonoidalSite}{Theorem 3.10, \cite{krishnan2023cubicalII}}

\theoremstyle{plain}
\newtheorem*{thm:degeneracies.unneeded}{Theorem \ref{thm:degeneracies.unneeded}}
\newtheorem*{prop:dold-kan.intro.case}{Proposition \ref{prop:dold-kan} [Case $\BOX=\PARSIMONIOUSBOX$]}
\newtheorem*{prop:cubical.homology.intro.case}{Proposition \ref{prop:cubical.homology} [Case $\BOX=\PARSIMONIOUSBOX$]}
\newtheorem*{prop:cubical.homology}{Proposition \ref{prop:cubical.homology}}
\newtheorem*{prop:classical.model.structure.intro.case}{Proposition \ref{prop:classical.model.structure} [Case $\BOX=\PARSIMONIOUSBOX$]}
\newtheorem*{thm:proper.test.intro.case}{Theorem \ref{thm:proper.test} [Case $\BOX=\PARSIMONIOUSBOX$]}
\newtheorem*{prop:Grothendieck.test.criterion}{\PROPGrothendieckTestCriterion}
\newtheorem*{thm:minimal.symmetric.monoidal.site}{\THMMinimalSymmetricMonoidalSite{}}
\newtheorem*{thm:semilattice.birkhoff}{\PROPSemilatticeBirkhoff}
\newtheorem*{thm:boolean.duality}{Theorem, \cite{church1940numerical}}
\newtheorem*{thm:birkhoff.duality}{\THMBirkhoffDuality}
\newtheorem*{thm:cat.zero}{\THMCATZero}
\newtheorem*{lem:box.automorphisms}{\LEMBoxAutomorphisms}
\newtheorem*{lem:box.epimorphisms}{\LEMBoxEpimorphisms}
\newtheorem*{prop:strict.test}{Corollary 2, \cite{buchholtz2017varieties}}
\newtheorem*{prop:connection.cube.maps}{Proposition 2.3, \cite{maltsiniotis2009categorie}}
\newtheorem*{thm:lattice.box.characterization}{Theorem 3, \cite{krishnan2023cubicalII}}
\newtheorem*{thm:main}{Theorem {\ref{thm:main}}}
\newtheorem*{cor:formula}{Corollary {\ref{cor:formula}}}
\newtheorem*{cor:sd}{Corollary {\ref{cor:sd}}}
\newtheorem*{cor:gammas}{Corollary {\ref{cor:gammas}}}
\newtheorem*{thm:birkhoff}{\THMBirkhoff}
\newtheorem*{cor:ez.cellular.model}{\COREZCellularModel}
\newtheorem*{prop:transferred.regularity}{\PROPTransferredRegularity}
\newtheorem*{prop:test.functor}{\PROPTestFunctor}
\newtheorem*{thm:old.quillen.equivalence}{\THMOldQuillenEquivalence}
\newtheorem*{thm:frobenius}{\THMFrobenius}
\newtheorem*{thm:algebraic.characterization}{Theorem \ref{thm:algebraic.characterization}}
\newtheorem*{prop:geometric.characterization}{Proposition \ref{prop:geometric.characterization}}

\newcommand{\THMSimplicialClassicalEquivalence}{Theorem \S2.3, \cite{quillen1967homotopical}}
\newtheorem*{thm:simplicial.classical.equivalence}{\THMSimplicialClassicalEquivalence}

\newcommand{\EGSSets}{Examples 3.6 and 4.9, \cite{gambino2017frobenius}}
\newtheorem*{eg:ssets}{\EGSSets}

\newtheorem*{eg:m-spaces}{\EGMSpaces}
\newtheorem*{lem:mincubical.monics}{\LEMMincubicalMonics}
\newtheorem*{thm:x-closed}{\THMXClosed}
\newtheorem*{thm:pointwise}{\THMPointwise}
\newtheorem*{thm:transfer}{\THMTransfer}
\newtheorem*{thm:small-object-argument}{\THMTransfer}

\newtheorem*{thm:cat}{\THMCat}
\newtheorem*{lem:sd.tri}{\LEMSdTri}
\newtheorem*{lem:qt-cocontinuous}{\LEMQTCocontinuous}
\newtheorem*{prop:underlying.monoidal.preordered.sets}{\PROPUnderlyingMonoidalPreorderedSets}
\newtheorem*{prop:locally.compact.streams}{\PROPLocCompactStreams}
\newtheorem*{thm:pospaces}{\THMPospaces}
\newtheorem*{thm:diembed}{\ThmDiEmbed}
\newtheorem*{thm:van-kampen}{\THMVanKampen}
\newtheorem*{thm:q-cubical}{\THMQCubical}
\newtheorem*{thm:tri.equivalence}{\THMTriEquivalence}
\newtheorem*{cor:tri.equivalence}{\CORTriEquivalence}
\newtheorem*{thm:realize.equivalence}{\THMRealizeEquivalence}
\newtheorem*{thm:ams}{\THMAMS}
\newtheorem*{eg:mixing}{Example 2.2 from \cite{cole2006mixing}}
\newtheorem*{thm:classical-equivalence}{\THMClassicalEquivalence}
\newtheorem*{prop:hp}{Proposition \ref{prop:hp}}
\newtheorem*{cor:types}{Corollary \ref{cor:types}}
\newtheorem*{prop:extensions}{Proposition \ref{prop:extensions}}
\newtheorem*{lem:geodesic.metric}{Lemma 3.70, \cite{goubault2020directed}}
\newtheorem*{lem:simplicial.graphs}{\SimplicialGraphs}

\newtheorem*{cor:algebraic-q-cubical}{Corollary \ref{cor:algebraic-q-cubical}}
\newtheorem*{cor:algebraic-m-spaces}{Corollary \ref{cor:algebraic-m-spaces}}

\newtheorem*{thm:approx}{Theorem {\ref{thm:approx}}}
\newtheorem*{thm:m-streams}{Theorem {\ref{thm:m-streams}}}
\newtheorem*{thm:cubical.dihomotopy}{Theorem {\ref{thm:cubical.dihomotopy}}}
\newtheorem*{cor:cubical.dihomotopy}{Corollary {\ref{cor:cubical.dihomotopy}}}

\newtheorem*{prop:cubical.homotopies}{Proposition {\ref{prop:cubical.homotopies}}}
\newtheorem*{prop:nerve.cubcats}{Proposition {\ref{prop:nerve.cubcats}}}
\newtheorem*{prop:kan.cubcat}{Proposition {\ref{prop:kan.cubcat}}}
\newtheorem*{cor:kan.bicubcat}{Corollary {\ref{cor:kan.bicubcat}}}
\newtheorem*{thm:fibrant}{Theorem {\ref{thm:fibrant}}}
\newtheorem*{prop:semifibrant}{Proposition {\ref{prop:semifibrant}}}
\newtheorem*{thm:equivalence}{Theorem {\ref{thm:equivalence}}}
\newtheorem*{cor:type-theoretic}{Corollary {\ref{cor:type-theoretic}}}
\newtheorem*{thm:q-spaces}{Theorem from \cite[\S II.3]{quillen2006homotopical}}
\newtheorem*{cor:q-equivalences}{Corollary {\ref{cor:q-equivalences}}}
\newtheorem*{cor:m-equivalences}{Corollary {\ref{cor:m-equivalences}}}
\newtheorem*{cor:excision}{Corollary {\ref{cor:excision}}}
\newtheorem*{cor:equivalence}{Corollary {\ref{cor:equivalence}}}
\newtheorem*{cor:cubical.diequivalence}{Corollary {\ref{cor:cubical.diequivalence}}}
\newtheorem*{thm:cartesian.closed.streams}{\THMXClosed}
\newtheorem*{thm:whitehead}{Theorem \ref{thm:whitehead}}
\newtheorem*{prop:topological}{\PropTop}
\newtheorem*{prop:sd}{\PropSD}
\newtheorem*{prop:inclusions}{\PROPInclusions}
\newtheorem*{thm:d-homotopy}{\ThmDHomotopy}
\newtheorem*{lem:box-inclusions}{\LEMBoxInclusions{}}

\title{A convenient category of cubes}
\author{Sanjeevi Krishnan and Emily Rudman}

\begin{document}
\begin{abstract}
  We claim that the cube category whose morphisms are the interval-preserving monotone functions between finite Boolean lattices is a convenient general-purpose site for cubical sets.
  This category is the largest possible concrete Eilenberg-Zilber variant excluding the reversals and diagonals. 
  The category admits as monoidal generators all functions between the ordinals $[0]$ and $[1]$ and all monotone surjections $\boxobj{n}\ra[1]$.
  Consequently, morphisms in the minimal symmetric monoidal variant of the cube category containing coconnections of one kind can be characterized as the interval-preserving semilattice homomorphisms between finite Boolean lattices.
  There exists a model structure on our variant of cubical sets that is at once Quillen equivalent to and left induced from the classical model structure on simplicial sets along triangulation.
  This model structure is proper and hence its fibrations interpret Martin-Lof dependent types.  
\end{abstract}

\maketitle
\tableofcontents
\addtocontents{toc}{\protect\setcounter{tocdepth}{1}}

\section{Introduction}
Spaces are often combinatorially represented as simplicial or cubical sets.
A cubical representation has some advantages over a simplicial representation.
One is that homotopy invariants like cup products and triad homotopy groups are more naturally described in terms of cubes \cite{huebschmann2012ronald} than simplices \cite{verity2007weak2,verity2007weak1}. 
Another is that cubical sets, unlike simplicial sets \cite[Remark A.3]{krishnan2022uniform}, admit approximation theorems with strong claims of naturality (eg. \cite{krishnan2015cubical,krishnan2022uniform}).
Another is that some fundamental axioms in synthetic simplicial homotopy theory become theorems in synthetic cubical homotopy theory \cite{bezem2014model}.  
Yet another is that non-positive curvature is easier to characterize in cubical \cite{gromov1987hyperbolic} than in simplicial (cf. \cite{januszkiewicz2006simplicial}) settings.    
Just as simplicial sets are presheaves over the category $\DEL$ of non-empty finite ordinals $[0],[1],[2],\ldots$, cubical sets are presheaves over a category $\BOX$ of the finite Boolean lattices
\begin{equation}
  \label{eqn:boolean.lattices}
[0],[1],[1]^2,[1]^3,\ldots
\end{equation}

The problem is that the definition of $\BOX$-morphisms is not standard \cite{buchholtz2017varieties,grandis2003cubical}.
The $\DEL$-morphisms can be both explicitly characterized as the monotone functions between non-empty finite ordinals $[0],[1],[2],\ldots$ and implicitly characterized by a presentation in terms of cofaces and codegeneracies \cite{may1992simplicial}. 
The original variant $\DEL_1^*$ of $\BOX$ in the literature, the free monoidal category with unit $[0]$ generated by the category $\DEL_1$ of ordinals $[0]$ and $[1]$ and all functions between them, has a simple presentation in terms of cofaces and codegeneracies \cite[Theorem 4.2]{grandis2003cubical}.
But unlike $\DEL$-morphisms, $\DEL_1^*$-morphisms do not admit a short and explicit characterization.  
And formal properties of the simplex category $\DEL$ convenient for combinatorial homotopy are lacking in the original cube category $\DEL_1^*$.
Recent years have seen an explosion of variants for $\BOX$ proposed to address different shortcomings of $\DEL_1^*$ (eg. \cite{angiuli2021syntax,buchholtz2017varieties,grandis2003cubical,isaacson2011symmetric,maltsiniotis2009categorie}).

The current situation in cubical theory resembles an earlier situation in topology.
Various technical lemmas in homotopy theory, at their natural levels of generality, called for various fussy point-set axioms on a topological space.  
For some variants of $\BOX$, finite Cartesian products of cubical sets model finite Cartesian products of the spaces they represent up to homotopy equivalence \cite{maltsiniotis2009categorie}.  
For some variants of $\BOX$, a cubical homotopy theory can be defined that models datatypes in certain higher order programming languages \cite{angiuli2021syntax,cavallo2020unifying}. 
For some variants of $\BOX$, cubical sets can be used to model homotopy theories of spaces with extra geometry, such as a metric \cite{goubault2020directed,krishnan2022uniform}, uniformity \cite{krishnan2022uniform}, or directionality \cite{goubault2020directed,krishnan2023cubicalII}.  
For some variants of $\BOX$, the natural tensor product on cubical sets is symmetric and thereby convenient for various applications \cite{grandis2009role,isaacson2011symmetric}.
Nowdays, a topological space in homotopy theory is typically redefined to just be a (weak) Hausdorff k-space so as to avoid having to make different point-set assumptions for different applications.
In this same spirit, we propose a novel variant $\PARSIMONIOUSBOX$ for $\BOX$ whose objects are the finite Boolean lattices (\ref{eqn:boolean.lattices}) and whose morphisms are the \textit{interval-preserving monotone functions} as a general-purpose choice for $\BOX$. 

This variant occupies a middle ground.
On one hand, the $\PARSIMONIOUSBOX$-morphisms include but are not generated by cofaces, codegeneracies, coordinate permutations, and codegeneracy-like maps called \textit{coconnections}.  
On the other hand, the $\PARSIMONIOUSBOX$-morphisms exclude all \textit{reversals} and \textit{diagonals}.
The exclusion of reversals allows cubical sets to admit edge-orientations and thus admit \textit{directed topological realizations} (e.g. \cite{krishnan2023cubicalII}).  
The exclusion of diagonals allows $\BOX$-objects to model topological cubes with their $\ell_p$ metrics and hence allows cubical sets to admit \textit{$\ell_p$-realizations} (e.g. \cite{krishnan2022uniform}), not just for $p=\infty$ but for all $1\leqslant p\leqslant\infty$; $\ell_2$-realizations in particular generalize uniquely geodesic CAT(0) cubical complexes.
The category $\PARSIMONIOUSBOX$ is the largest reasonable variant of $\BOX$ excluding the reversals and diagonals:

\begin{prop:geometric.characterization}
  Consider a subcategory $\shape{1}$ of $\SETS$ such that the following all hold:
           \begin{enumerate}
		   \item $\shape{1}$ excludes the reversal $[1]\ra[1]$ and diagonal $[1]\ra[1]^2$
		   \item $\shape{1}$ contains $\DEL_1^*$ as a wide subcategory
		   \item every $\shape{1}$-morphism factors into a composite of a surjective $\shape{1}$-morphism followed by an injective $\shape{1}$-morphism.
  \end{enumerate}
  Then $\shape{1}$ is a subcategory of $\PARSIMONIOUSBOX$.
  The choice $\shape{1}=\PARSIMONIOUSBOX$ satisfies all three conditions above.  
\end{prop:geometric.characterization}

We can identify generators for $\PARSIMONIOUSBOX$ as a symmetric monoidal category.   
Unlike most variants of $\BOX$ in the literature, $\PARSIMONIOUSBOX$ does not admit a finite monoidal presentation.  
Also unlike most variants of $\BOX$ in the literature, $\PARSIMONIOUSBOX$-morphisms do not generally preserve any obvious algebraic structure on finite Boolean lattices.  
Nonetheless, $\PARSIMONIOUSBOX$-morphisms admit the following decomposition.
Birkhoff Duality allows us to reinterpret the endomorphism operad of $[1]$ in $\PARSIMONIOUSBOX$ as what we might call the \textit{distributive lattice operad}, the operad of free and finite distributive lattices, where the operadic action substitutes generators from such a lattice with terms from other such lattices (c.f. \cite[Exercise 2.2.11]{leinster2004higher}).
The symmetric monoidal category $\PARSIMONIOUSBOX$ is almost generated by this operad but for the presence of the terminal function $[1]\ra[0]$, by the following splitting result.

\begin{thm:algebraic.characterization}
  The category $\PARSIMONIOUSBOX$ is generated as a symmetric monoidal category by \ldots
  \begin{enumerate}
    \item the unique function $\sigma:[1]\twoheadrightarrow[0]$
    \item the two functions of the form $[0]\ra[1]$
    \item all monotone surjections $\boxobj{n}\twoheadrightarrow[1]$
  \end{enumerate}
	For each $\PARSIMONIOUSBOX$-morphism $\phi:\boxobj{m}\ra\boxobj{n}$, there exists a unique $n$-tuple $(m_1,m_2,\ldots,m_n)$ of natural numbers with $m_1+m_2+\cdots+m_n\leqslant m$ and unique coset in $\Sigma_m/(\Sigma_{m_1}\times\cdots\times\Sigma_{m_n})$ such that for each representative $g\in\Sigma_m$ of that coset, there exist unique monotone functions $\phi_i:\boxobj{m_i}\ra[1]$ for $1\leqslant i\leqslant n$, each non-constant in each of its coordinates, with $\phi g=\phi_1\otimes\cdots\otimes\phi_n\otimes\sigma^{\otimes(m-m_1-\cdots-m_n)}$.
\end{thm:algebraic.characterization}

This result has direct applications.  
The first is a computational implementation of cubical sets when $\BOX=\PARSIMONIOUSBOX$.   
The theorem suggests a representation of $\PARSIMONIOUSBOX$-morphisms in terms of natural numbers, formal expressions involving lattice operations, and permutations.  
The action of the distributive lattice operad and a solution to the word problem for free distributive lattices \cite{takeuchi1969word} can be used to give an algorithm for composition in $\PARSIMONIOUSBOX$ in terms of those representations.
The second is a short, explicit characterization of the morphisms in the minimal symmetric monoidal variants of $\BOX$ (e.g. \cite[Theorem 3.10]{krishnan2023cubicalII}) containing coconnections of one kind [Corollaries \ref{cor:join.characterization} and \ref{cor:meet.characterization}].
Short, explicit characterizations for $\BOX$-morphisms make it easy to functorially construct finite CAT(0) cubical complexes from a nerve-like construction [Proposition \ref{prop:boolean.functions}, Corollary \ref{cor:sd} and last paragraph of \S\ref{sec:cat0.cubical.complexes}] based on observations made elsewhere \cite{ardila2012geodesics,gonzalez2021finite}.   

The most common application of cubical sets is a combinatorial description of classical homotopy theory.  
We give such a description in the form of a \textit{classical model structure} on cubical sets $\hat\BOX$, left induced along topological realization.
In the statement below, existence follows from general abstract arguments \cite[Theorem 2.2.1]{hess2017necessary} but the equivalence is not entirely straightforward.  

\begin{prop:classical.model.structure.intro.case}
  There exists a model structure on $\hat\PARSIMONIOUSBOX$ in which \ldots
  \begin{enumerate}
	  \item \ldots a weak equivalences $\psi$ is characterized by $|\psi|$ a homotopy equivalence
      \item \ldots the cofibrations are the monos
  \end{enumerate}
  Topological realization defines the left map of a Quillen equivalence from $\hat\PARSIMONIOUSBOX$ equipped with this model structure to the category of topological spaces equipped with its usual model structure.
\end{prop:classical.model.structure.intro.case}

A major source of recent interest in cubical sets is their potential to constructively model datatypes in higher order programming languages \cite{anguili2018cartesian,awodey2018cubical, awodey2024equivariant,mortberg2020cubical}.  
In a right proper model structure on a presheaf category in which the cofibrations are the monos, fibrations model datatypes parametrized by other datatypes, (fibered) equivalences model equivalences of (parametrized) datatypes, and right properness models substitutability of one parameter with an equivalent parameter \cite{shulman2015univalence}.
Right properness for simplicial sets follows from the fact that geometric realizations of Kan fibrations are Serre fibrations \cite{quillen1968geometric}.
Just enough right lifting properties are preserved by geometric realizations [Lemma \ref{lem:dold.fibration}] to make the essence of that argument carry through to the cubical setting $\hat\BOX$, at least when $\BOX$ is $\PARSIMONIOUSBOX$ or certain subvariants. 

\begin{thm:proper.test.intro.case}
  The classical model structure on $\hat\PARSIMONIOUSBOX$ is proper.
\end{thm:proper.test.intro.case}

Each choice for $\BOX$ requires tradeoffs.
Firstly, cubical groups defined by $\BOX=\BOX_c$, like simplicial groups, are automatically fibrant in classical/test model structures. 
It is not clear whether that fibrancy continues to hold for symmetric monoidal extensions $\BOX$ of $\BOX_c$ like $\PARSIMONIOUSBOX$.  
Secondly, a cubical Dold-Kan equivalence for $\BOX=\BOX_c$ no longer straightforwardly (c.f. \cite{lack2015combinatorial}) holds when $\BOX$ is extended to larger variants.
But the inclusion $\BOX_c\subset\PARSIMONIOUSBOX$ allows, for example, a recent construction of representing objects for bounded cubical cohomology on connected cubical sets \cite{krishnan2022uniform}, based on the Dold-Kan equivalence, to adapt to the setting $\BOX=\PARSIMONIOUSBOX$.  
Thirdly, implementations of synthetic homotopy theory call for explicit choices of generating acyclic cofibrations in \textit{type-theoretic} \cite{shulman2015univalence} model structures.  
Such choices, available for several variants of $\BOX$, are not yet available for $\PARSIMONIOUSBOX$.  
Lastly, diagonals appear important in recent cubical interpretations of \textit{higher inductive types}, implementations of homotopy colimits as datatype constructors. 
But the exclusion of those same diagonals appear important in current proofs of cubical approximation \cite{krishnan2015cubical,krishnan2022uniform,krishnan2023cubicalII} with respect to homotopy theories in which homotopy colimit decompositions are rare.  

\subsection*{Organization}
The first half culminates in different characterizations of $\PARSIMONIOUSBOX$ [Proposition \ref{prop:geometric.characterization} and Theorem \ref{thm:algebraic.characterization}] as well as two subvariants [Corollaries \ref{cor:join.characterization} and \ref{cor:meet.characterization}]. 
The second half culminates in different observations about cubical homotopy theory [Proposition \ref{prop:classical.model.structure} and Theorem \ref{thm:proper.test}] not only for $\BOX=\PARSIMONIOUSBOX$ but also for other variants of $\BOX$ where possible.  
Adaptations of cubical approximation beyond the setting of classical homotopy for $\BOX=\PARSIMONIOUSBOX$ are mostly straightforward but not included in this paper.  
Along the way, a nerve-like construction of finite CAT(0) cubical complexes, based on observations made elsewhere \cite{ardila2012geodesics,gonzalez2021finite}, is described in \S\ref{sec:cat0.cubical.complexes}. 

\section{Conventions}\label{sec:conventions}
This section first fixes some conventions.
Let $k,m,n,p,q$ denote natural numbers.
Let $\I$ denote the unit interval.
Let $\ira$ denote an inclusion of some sort, such as an inclusion of a subset into a set, a subspace into a space, or a subcategory into a category.
A chain complex is \textit{connective} if it is non-negatively graded.  
The following symbols will denote certain distinguished maps, defined in \S\ref{sec:cubes}, in various cube categories: $\codeg,\coface_{\pm},\coconnect_{\pm},\transpose,\diagonal,\reversal$.  

\subsection{Categories}
Let $\cat{1},\cat{2}$ denote arbitrary categories.
Let $\mathscr{A}$ denote an Abelian category. 
Let $\shape{1}$ denote a small category.
Let $\star$ denote a terminal object in a given category.
For a given monoidal category, let $\otimes$ denote its tensor product.
Notate special categories as follows.
\vspace{.1in}\\
\begin{tabular}{rll}
  $\SETS$ & sets (and functions)\\
  $\DISLATS$ & finite distributive lattices (and lattice homomorphisms) \\
  $\POSETS$ & finite posets (and monotone functions) \\
  $\CATS$ & small categories (and functors) \\
  $\DEL$ & non-empty finite ordinals and monotone functions & \S\ref{sec:simplices}\\
  $\DEL_1$ & the full subcategory of $\SETS$ having objects $[0],[1]$ & \S\ref{sec:simplices} \\ 
  $\PARSIMONIOUSBOX$ & proposed variant of the cube category & \S\ref{sec:cubes}\\ 
\end{tabular}

\vspace{.1in}

We often regard $\SETS$ and $\POSETS$ as closed Cartesian monoidal categories.  
We write $\BOX$ to denote an operating definition of cubes, over which \textit{cubical sets} are defined as $\SETS$-valued presheaves.
Write $\hat{\shape{1}}$ for the category of presheaves over $\shape{1}$, the functor category
$$\hat{\shape{1}}=\SETS^{\OP{\shape{1}}}.$$

Write $\shape{1}[-]$ for the Yoneda embedding $\shape{1}\ira\hat{\shape{1}}$.  
For each diagram $\zeta:\shape{1}\ra\cat{1}$ to a cocomplete category $\cat{1}$, write $\zeta_!\dashv\zeta^*$ for the adjunction $\hat{\shape{1}}\lras\cat{1}$ whose left adjoint $\zeta_!$ is the left Kan extension of $\zeta$ along $\shape{1}[-]:\shape{1}\ra\hat{\shape{1}}$ and whose right adjoint $\zeta^*$ naturally sends each $\cat{1}$-object $o$ to $\cat{1}(\zeta(-),o)$.  
Let $F/G$ denote the comma category for diagrams $F,G$ in the same category.
For a diagram $F$ in $\hat{\shape{1}}$, let $\shape{1}/F=\shape{1}[-]/F$.
Let $\id_o$ denote the identity morphism for an object $o$ in a given category.
For each object $o$ in a given closed monoidal category, $o^{(-)}$ will denote the right adjoint to the endofunctor $o\otimes-$.
 
\subsection{Posets}
Write $\leqslant_P$ for the partial order on a poset $P$.
Write $[x,z]_P$ for the subposet 
$$[x,z]_P=\{y\in P\;|\;x\leqslant_Py\leqslant_Pz\}$$
of a poset $P$ and call $[x,z]_P$ an \textit{interval in $P$} if it is non-empty. 
In a poset $P$, an element $z$ is an \textit{immediate successor} to an element $x$ if $x\leqslant_Pz$ and $x=y$ or $y=z$ whenever $x\leqslant_Py\leqslant_Pz$.  
A function $\phi:P\ra Q$ between posets is \textit{monotone} if $\phi(x)\leqslant_Q\phi(y)$ whenever $x\leqslant_Py$.
Write $\POSETS$ for the category of posets and monotone functions between them.
Let $[n]$ denote the poset $\{0,1,\ldots,n\}$ equipped with the usual order.

\subsection{Semilattices}
A \textit{meet-semilattice} is a poset having all binary infima. 
Write $\wedge_L$ for the binary inifimum operation $\wedge_L:L^2\ra L$ on a meet-semilattice $L$.
A \textit{join-semilattice} is a poset having all binary suprema.  
Write $\vee_L$ for the binary supremum operation $\vee_L:L^2\ra L$ on a join-semilattice $L$.
A meet-semilattice homomorphism is a function between meet-semilattices preserving binary infima.  
A join-semilattice homomorphism is a function between join-semilattices preserving binary suprema.  

\subsection{Lattices}
A \textit{lattice} is always taken in the order-theoretic sense to mean a poset that is at once a meet-semilattice and join-semilattice.
Henceforth write $\multiboxobj{k}{n}$ for the $n$-fold $\POSETS$-product of $[k]$. 
A lattice is \textit{distributive} if the following holds for each $x,y,z\in L$:
$$x\wedge_L(y\vee_Lz)=(x\wedge_Ly)\vee_L(x\wedge_Lz)$$

\textit{Birkhoff Duality} refers to the categorical equivalence
$$\OP{\POSETS}\simeq \DISLATS$$
naturally sending each finite poset $P$ to $\POSETS(P,[1])$ with lattice operations defined element-wise from $[1]$.
A poset is \textit{Boolean} (and hence a \textit{Boolean lattice}) if it is a distributive lattice whose maximum is the supremum of the immediate successors to its minimum.

\begin{eg}
  The finite Boolean lattices are, up to $\POSETS$-isomorphism,
  $$[0],[1],\boxobj{2},\boxobj{3},\ldots$$
\end{eg}

Every interval in a Boolean lattice is Boolean. 
A function $\phi:P\ra Q$ between the underlying sets of posets \textit{preserves (Boolean) intervals} if images of (Boolean) intervals in $P$ under $\phi$ are (Boolean) intervals in $Q$.
A \textit{lattice homomorphism} is a function $\phi:L\ra M$ between lattices preserving binary suprema and binary infima.  
A \textit{free distributive lattice} is a lattice in the essential image of the left adjoint to the forgetful functor from the category of distributive lattices and lattice homomorphisms to $\SETS$.

\subsection{Groups}
Write $\Sigma_n$ for the symmetric group, the group of bijections on $\{1,2,\ldots,n\}$.  
We will often identify $\Sigma_n$ with the group of coordinate permutations on $[1]^n$.
Under this identification, we regard $\Sigma_m$ as acting on hom-sets $\POSETS(\boxobj{m},\boxobj{n})$ on the right: $\phi g$ denotes the composite of the coordinate permutation $g\in\Sigma_m$ on $\boxobj{m}$ followed by the monotone function $\phi:\boxobj{m}\ra\boxobj{n}$.  

For all $m\leqslant n$, we regard $\Sigma_m$ as the subgroup of $\Sigma_n$ consisting of all permutations that fix the last $n-m$ numbers.  
For all natural numbers $m_1+m_2+\cdots+m_n=m$, we regard $\Sigma_{m_1}\times\cdots\times\Sigma_{m_n}$ as the subgroup of $\Sigma_{m_1+\cdots+m_n}$ consisting of all the permutations which restrict and corestrict to permutations on $\{m_1+\cdots+m_{k-1}+1,\ldots,m_1+\cdots+m_k\}$ for each $1\leqslant k\leqslant n$.  
In this manner, we will implicitly treat $\Sigma_{m_1}\times\cdots\times\Sigma_{m_n}$ as a subgroup of $\Sigma_m$ if $m_1+m_2+\cdots+m_n\leqslant m$.

\addtocontents{toc}{\protect\setcounter{tocdepth}{2}}

\section{Shapes}\label{sec:shapes}
We recall the simplex category in \S\ref{sec:simplices} and common, existing variants of the cube category (cf. \cite{isaacson2011symmetric}) in \S\ref{sec:existing.cube.categories}.  
We then introduce and investigate our proposed variant $\PARSIMONIOUSBOX$ of the cube category in \S\ref{sec:proposed.cube.category}.
We first fix some notation.  
For each small subcategory $\shape{1}$ of $\SETS$, we write $\shape{1}_{\pls}$ and $\shape{1}_{\mins}$ for the wide subcategories of $\shape{1}$ consisting solely of the respectively injective and surjective $\shape{1}$-morphisms.

\subsection{Simplices}\label{sec:simplices}.  
Let $\DEL$ denote the category of non-empty finite ordinals 
$$[0],[1],[2],\ldots$$
and all monotone functions between them.
We write $\oplus$ for the \textit{ordinal sum} \cite{ehlers2008ordinal} bifunctor 
$$\oplus:\DEL\times\DEL\ra\DEL$$
defined on each pair $(\alpha:[a_1]\ra[a_2],\beta:[b_1]\ra[b_2])$ of morphisms by commutative diagrams 
\begin{equation*}
  \begin{tikzcd}
	  {[a_1]\amalg[b_1]}\ar{rrrrrr}[above]{([a_2]\ira[a_2]\amalg[b_2])\alpha\amalg([b_2]\ira[a_2]\amalg[b_2])\beta}\ar{d}[left]{(i\mapsto i)\amalg(i\mapsto a_1+1+i)} & & & & & & {[a_2]\amalg[b_2]}\ar{d}[right]{(i\mapsto i)\amalg(i\mapsto a_2+1+i)}\\
	  {[a_1+b_1+1]}\ar{rrrrrr}[below]{\alpha\oplus\beta} & & & & & & {[a_2+b_2+1]}
  \end{tikzcd}
\end{equation*}
in $\POSETS$.
Write $\DEL_n$ for the full subcategory of $\DEL$ whose objects are $[0],[1],\ldots,[n]$.

\begin{eg}
  The category $\DEL_1$ is the subcategory of $\SETS$ generated by 
  $$\coface_{\mins}:[0]\ra[1],\coface_{\pls}:[0]\ra[1],\codeg:[1]\ra[0].$$

  Iterated ordinal sums of identities with $\coface_{\pm}$ are usually referred to as \textit{cofaces}.
  Iterated ordinal sums of identities with $\codeg$ are usually referred to as \textit{codegeneracies}.
  We will reserve that terminology for monotone functions similarly defined, but with $\SETS$-Cartesian monoidal products playing the role of ordinal sums.  
\end{eg}

\subsection{Cubes}\label{sec:cubes}
All variants of $\BOX$ are monoidal categories whose objects are the lattices 
$$[0],[1],\boxobj{2},\boxobj{3},\ldots$$
such that $[0]$ is the unit of the tensor product.
Most choices of $\BOX$ are monoidal subcategories of the Cartesian monoidal category $\SETS$ (cf. \cite{isaacson2011symmetric}).  
And several choices of $\BOX$ are additionally subcategories of $\POSETS$.

\subsubsection{Existing variants}\label{sec:existing.cube.categories}
For each small category $\shape{1}$ having as its objects $[0],[1],\boxobj{2},\boxobj{3},\ldots$, write $\shape{1}_n$ for the full subcategory of $\shape{1}$ containing $[0],[1],\boxobj{2},\ldots,\boxobj{n}$.  
Let $\DEL_1[\phi_1,\phi_2,\ldots,\phi_n]$ denote the smallest subcategory of $\SETS$ containing $\DEL_1$ and the functions $\phi_1,\phi_2,\ldots,\phi_n$.  
For each subcategory $\shape{1}$ of $\SETS$, we write $\shape{1}^*$ for the minimal submonoidal subcategory of the Cartesian monoidal category $\SETS$ containing $\shape{1}$.  
Define the following functions
\begin{align*}
	\codeg&:[1]\ra[0]\\
	\coface_{\pm}&:[0]\ra[1] & \coface_{\pm}(0)&=\half\pm\half\\
	\transpose &:\boxobj{2}\cong\boxobj{2} & \transpose(x,y)&=(y,x)\\
	\coconnect_{\pls}&:\boxobj{2}\ra[1] & \coconnect_{\pls}(x,y) &=\max(x,y)\\
	\coconnect_{\mins}&:\boxobj{2}\ra[1] & \coconnect_{\mins}(x,y) &=\min(x,y)\\
	\diagonal&:[1]\ra\boxobj{2} & \diagonal(x) &=(x,x)\\
	\reversal &:[1]\cong[1] & \reversal(x) &= 1-x
\end{align*}

For each function $\phi:\boxobj{n_1}\ra\boxobj{n_2}$ and $1\leqslant i\leqslant n$, let $\phi_{i;n}$ denote the function 
$$\phi_{i;n}=\boxobj{i-1}\otimes\phi\otimes\boxobj{n-i}:\boxobj{n+n_1-1}\ra\boxobj{n+n_2-1},$$
where $\otimes$ denotes the Cartesian monoidal tensor on $\SETS$.  
\textit{Cofaces}, \textit{codegeneracies}, \textit{coconnections}, \textit{coreversals}, \textit{diagonals}, and \textit{coordinate transpositions} are functions of the above form.  

\textit{Cofaces} are monotone functions of the form 
$$\coface_{\pm i;n+1}=(\coface_{\pm})_{i;n+1}:\boxobj{n}\ra\boxobj{n+1}.$$

\begin{lem:box-inclusions}
  For each $n$ and interval $I$ in $\boxobj{n}$, there exist unique $m_I$ and composite 
  $$\boxobj{m_I}\ra\boxobj{n}$$
  of cofaces that has image $I$.
\end{lem:box-inclusions}

Cofaces are not enough to define a suitable variant of $\BOX$.

\begin{eg}
  Cofaces generate $(\DEL_{1}^*)_{\pls}=((\DEL_{1})_{\pls})^*$.  
  The $\SETS$-valued presheaves
  $$\OP{(\DEL_{1}^*)_{\pls}}\ra\SETS$$
  over $(\DEL_{1}^*)_{\pls}$ are generally regarded as \textit{precubical/semicubical sets} (eg. \cite{fajstrup2006algebraic}) and not cubical sets (cf. \cite{buchholtz2017varieties}).
\end{eg}

Variants of $\BOX$ should also include \textit{codegeneracies}, monotone functions of the form 
$$\codeg_{i;n+1}:\boxobj{n+1}\ra\boxobj{n}.$$

\begin{eg}
  The codegeneracy $\codeg_{i;n+1}$ is exactly the projection
  $$\codeg_{i;n+1}:\boxobj{n+1}\ra\boxobj{n}$$
  onto all but the $i$th factor.  
  Cogeneracies generate $(\DEL_{1}^*)_{\mins}=((\DEL_{1})_{\mins})^*$.  
\end{eg}

\begin{eg}
  \label{eg:cubes.with.coconnections.of.first.kind}
  The following are equivalent for a small category $\shape{1}$. 
  \begin{enumerate}
	  \item $\shape{1}=\DEL_1^*$
	  \item $\shape{1}$ is the free monoidal category with unit $[0]$ generated by $\DEL_1$
	  \item $\shape{1}$ is the subcategory of $\POSETS$ generated by all cofaces and codegeneracies
  \end{enumerate}
  This category $\BOX$ satisfying any of the above equivalent conditions is the original and minimal variant of $\BOX$ adopted in the literature \cite[Theorem 4.2]{grandis2003cubical}.
  An explicit presentation of $\DEL_1^*$ is given in the literature \cite[\S 4]{grandis2003cubical}.  
  Most variants of $\BOX$ are submonoidal subcategories of the Cartesian monoidal category $\SETS$ (cf. \cite{isaacson2011symmetric}) containing $\DEL_1^*$ as a wide subcategory.  
\end{eg}

\textit{Principal coordinate transpositions} are bijective $\BOX$-morphisms of the form 
$$\transpose_{i;n}:\boxobj{n+2}\cong\boxobj{n+2}.$$

\begin{eg}
  For functions $\phi_1,\phi_2,\ldots,\phi_n$ between $\DEL_1^*$-objects,
  $$\DEL_1[\transpose,\phi_1,\ldots,\phi_n]^*$$
  is the smallest symmetric monoidal subcategory of $\SETS$ containing $\DEL_1$ and $\phi_1,\phi_2,\ldots,\phi_n$.
  In particular, $\DEL_1[\transpose]^*$ is the minimal symmetric monoidal variant of the cube category.
\end{eg}

\begin{lem:box.automorphisms}
  The following are equivalent for a monotone function of the form
  $$\phi:\boxobj{m}\ra\boxobj{n}.$$
  \begin{enumerate}
	    \item\label{item:box.bijection} $\phi$ is bijective
	    \item\label{item:interval.preserving.lattice.isomorphism} $\phi$ is an interval-preserving bijection 
	    \item\label{item:automorphism} $\phi$ is a lattice isomorphism
	    \item\label{item:permutation} $\phi$ is a coordinate permutation
	    \item\label{item:cosymmetry} $\phi$ is composite of principal coordinate transpositions
  \end{enumerate}
\end{lem:box.automorphisms}

The lemma allows us to recognize the automorphism group of $\boxobj{n}$ in every subcategory $\BOX$ of $\POSETS$ containing $\DEL_1[\tau]^*$ as a wide subcategory as the group $\Sigma_n$ of coordinate permutations on $\boxobj{n}$.

\begin{lem:box.epimorphisms}
  The following are equivalent for a lattice homomorphism of the form
  $$\phi:\boxobj{m}\ra\boxobj{n}.$$
  \begin{enumerate}
	    \item\label{item:box.surjection} $\phi$ is surjective
	    \item\label{item:interval.preserving.lattice.isomorphism} $\phi$ is an interval-preserving surjection
	    \item\label{item:automorphism} $\phi$ is a composite of principal coordinate transpositions and codegeneracies
  \end{enumerate}
\end{lem:box.epimorphisms}

\textit{Coconnections} of the first and second kinds are monotone functions of the respective forms $\coconnect_{\pm i;n}=(\coconnect_{\pm})_{i;n}:\boxobj{n+1}\ra\boxobj{n}$.
Coconnections act like extra codegeneracies that give a cube category some of the test categorical properties of $\DEL$.  

\begin{eg}
  The following are equivalent for a small category $\shape{1}$. 
  \begin{enumerate}
	  \item $\shape{1}=\DEL_1[\gamma_{\mins}]^*$
	  \item $\shape{1}$ is the free monoidal category with unit $[0]$ generated by $\DEL_1[\gamma_{\mins}]$
	  \item $\shape{1}$ is the subcategory of $\POSETS$ generated by all cofaces, codegeneracies, and coconnections of the form $\gamma_{\mins i}=(\gamma_{\mins})_i$.
  \end{enumerate}
	This category $\BOX_c$ satisfying any of the above equivalent conditions is one of the most common variants of $\BOX$ studied in the literature and exhibits technical conveniences lacking in the original variant $\DEL_1^*$: cubical groups defined by $\BOX=\BOX_c$ are fibrant in the test model structure \cite[Theorem 2.1]{tonks1992cubical}; a Dold-Kan equivalence between categories of connective chain complexes and cubical Abelian objects holds when $\BOX=\BOX_c$ \cite[Theorem 14.8.1]{huebschmann2012ronald}; and localization by weak equivalences in the test model structure on $\hat\BOX$ preserves finite products when $\BOX=\BOX_c$ \cite[Proposition 4.3]{maltsiniotis2009categorie}.  
\end{eg}

\textit{Diagonals} are monotone functions of the form 
$$\diagonal_{i;n}:\boxobj{n+1}\ra\boxobj{n+2}.$$  

\begin{eg}
  The following are equivalent for a small category $\shape{1}$. 
  \begin{enumerate}
	  \item $\shape{1}=\DEL_1[\tau,\diagonal]^*$
	  \item $\shape{1}$ is the subcategory of $\POSETS$ generated by all cofaces, codegeneracies, principal coordinate transpositions, and diagonals.
  \end{enumerate}
  The category satisfying any of the above equivalent conditions is the \textit{Cartesian cube category}.
  The Cartesian cube category was proposed as a model of cubical sets convenient for synthetic cubical homotopy theory.   
\end{eg}

\begin{lem}
  \label{lem:diagonals}
  Fix a category $\shape{1}$ in a chain of inclusions of subcategories
  $$\DEL_1^*\subset\shape{1}\subset\POSETS$$
  with $\DEL_1^*$ wide in $\shape{1}$.  
  The following are equivalent:
  \begin{enumerate}
    \item $\shape{1}$-morphisms map $1$-dimensional intervals onto intervals
	    \item $\shape{1}$ does not contain the diagonal $[1]\ra[1]^2$
	  \end{enumerate}
	\end{lem}
	\begin{proof}
	  Note (1) implies (2) because the diagonal $[1]\ra[1]^2$ does not map $1$-dimensional intervals onto intervals.  

	  Suppose (1) is not true.  
	  Then there exists a $\shape{1}$-morphism $\phi:\boxobj{m}\ra\boxobj{n}$ not mapping a $1$-dimensional interval $I$ onto an interval.  
	  We can choose $\phi$ so that $m=1$ without loss of generality by precomposing $\phi$ with a composite $[1]\ra\boxobj{m}$ of cofaces, a $\shape{1}$-morphism, having image $I$ [\LEMBoxInclusions].
	  There exists a unique composite $\delta$ of cofaces whose image in the image of $\phi$ [\LEMBoxInclusions]
	  We can therefore take $\phi$ to be exrema-preserving by composing $\phi$ with the unique retraction of $\delta$ in $\DEL_1^*$, a $\shape{1}$-morphism.
	  Then $n>1$ by the image of $\phi$ not an interval.  
	  Thus there exists a composite $\pi$ of codegeneracies of the form $\boxobj{n}\ra\boxobj{2}$.  
	  Then $\pi\phi:[1]\ra\boxobj{2}$ is extrema-preserving by $\pi$ and $\phi$ extrema-preserving and hence $\pi\phi$ is the diagonal $[1]\ra\boxobj{2}$. 
	  Thus (2) is not true.
	\end{proof}

\textit{Reversals} are non-monotone bijections of the form 
	$$\reversal_{i;n}:\boxobj{n+1}\cong\boxobj{n+1}.$$

\begin{eg}
  The following are equivalent for a small category $\shape{1}$. 
  \begin{enumerate}
	  \item $\shape{1}=\DEL_1[\tau,\gamma_{\mins},\gamma_{\pls},\diagonal,\reversal]^*$
	  \item $\shape{1}$ is the subcategory of $\SETS$ generated by all cofaces, codegeneracies, coconnections of the first kind, coconnections of the second kind, principal coordinate transpositions, reversals and diagonals.
  \end{enumerate}
  The category satisfying any of the above equivalent conditions is the \textit{de Morgan cube category}.
  The de Morgan cube category was proposed as a model of cubical sets convenient for synthetic cubical homotopy theory.   
\end{eg}

	\begin{lem}
	  \label{lem:reversals}
	  For a small category $\shape{1}$ in a chain of inclusions of subcategories
	  $$\DEL_1^*\subset\shape{1}\subset\SETS.$$
	  with $\DEL_1^*$ wide in $\shape{1}$, the following are equivalent:
	  \begin{enumerate}
	    \item $\shape{1}$-morphisms are monotone functions between the lattices $[0],[1],\boxobj{2},\ldots$
	    \item $\shape{1}$ does not contain the reversal $[1]\cong[1]$
	  \end{enumerate}
	\end{lem}
	\begin{proof}
	  Note (1) implies (2) because $\reversal:[1]\ra[1]$ is not monotone.  

	  Suppose (1) is not true.  
	  Then there exists a $\shape{1}$-morphism $\phi:\boxobj{m}\ra\boxobj{n}$ that is not monotone.  
	  Therefore there exist $x\leqslant_{\boxobj{m}}y$ such that $\phi(x)\nleqslant_{\boxobj{n}}\phi(y)$.  
	  We can take $y$ to be an immediate successor to $x$ in $\boxobj{m}$ because arrows between an element and its immediate successor generate $\boxobj{m}$ as a category.
	  Then $\{x,y\}$ is an interval in $\boxobj{m}$.  
	  We can thus choose $\phi$ so that $m=1$ without loss of generality by precomposing $\phi$ with a composite $[1]\ra\boxobj{m}$ of cofaces, a $\shape{1}$-morphism, having image $\{x,y\}$ [\LEMBoxInclusions].

	  There exists $1\leqslant i\leqslant n$ such that $\phi(x)_i=1$ and $\phi(y)_i=0$.  
	  Therefore we can take $n=1$ by composing $\phi$ with projection $\boxobj{n}\ra[1]$ onto the $i$th factor, a composite of codegeneracies and hence a $\shape{1}$-morphism.  
	  Then $\phi:[1]\ra[1]$ is not monotone and hence $\phi$ is the reversal $\reversal:[1]\cong[1]$.
	  Thus (2) is not true.
\end{proof}

\begin{eg}
  The following are equivalent for a small category $\shape{1}$. 
  \begin{enumerate}
    \item $\shape{1}=\DEL_1[\tau,\gamma_{\mins},\gamma_{\pls},\diagonal]^*$
    \item $\shape{1}$ is the subcategory of $\POSETS$ generated by all cofaces, codegeneracies, coconnections of the first kind, coconnections of the second kind, principal coordinate transpositions, and diagonals
    \item $\shape{1}$ is the full subcategory of $\POSETS$ whose objects are $[0],[1],\boxobj{2},\ldots$
     \item $\shape{1}$ is the maximal subcategory of $\SETS$ containing $\DEL_1^*$ as a wide subcategory and excluding $\reversal$
  \end{enumerate}
  The category $\shape{1}$ satisfying the above has been sometimes referred to in the recent literature \cite{cavallo2022relative,sattler2018idempotent} as the \textit{Dedekind cube category} because it was the first variant of $\BOX$ considered in which the values of the endomorphism operad of $[1]$ have the Dedekind numbers as their cardinalities.
  Birkhoff Duality restricts and corestricts to a categorical equivalence between $\OP{(\DEL_1[\tau,\gamma_{\mins},\gamma_{\pls},\diagonal]^*)}$ and the full subcategory of $\DISLATS$ consisting of the free and finite distributive lattices.
\end{eg}

\subsubsection{Proposed variant}\label{sec:proposed.cube.category}
Let $\PARSIMONIOUSBOX$ denote a novel variant of $\BOX$, the symmetric monoidal subcategory of the Cartesian monoidal category $\POSETS$ whose objects are the posets $[0],[1],\boxobj{2},\boxobj{3},\ldots$ and whose morphisms are all \textit{interval-preserving} monotone functions between them.

\begin{eg}
  For each $n$, we have the identities
	$$\PARSIMONIOUSBOX(\boxobj{n},[1])=\DEL_1[\transpose,\coconnect_{\mins},\coconnect_{\pls},\diagonal]^*(\boxobj{n},[1])=\POSETS(\boxobj{n},[1]).$$
\end{eg}

The variant $\PARSIMONIOUSBOX$ of $\BOX$ excludes diagonals and reversals and contains but is not generated by all cofaces, codegeneracies, coconnections of both kinds, and principal coordinate transpositions.    
Note that $\PARSIMONIOUSBOX_2=\DEL_1[\transpose,\coconnect_{\mins},\coconnect_{\pls}]$.
Moreover, only $\diagonal$ among the generators $\transpose,\coconnect_{\mins},\coconnect_{\pls},\diagonal$ of the full subcategory $\DEL_1[\transpose,\coconnect_{\mins},\coconnect_{\pls},\diagonal]^*$ of $\POSETS$ fails to define a $\PARSIMONIOUSBOX$-morphism.
Nonetheless, $\PARSIMONIOUSBOX\neq\DEL_1[\transpose,\coconnect_{\mins},\coconnect_{\pls}]^*$. 
 
\begin{eg}
  The following is a $\PARSIMONIOUSBOX$-morphism but not a $(\PARSIMONIOUSBOX_2)^*$-morphism:
  $$(x,y,z)\mapsto (x\wedge_{[1]}y)\vee(x\wedge_{[1]}z)\vee(y\wedge_{[1]}z):\boxobj{3}\ra[1].$$
\end{eg}

An analogue of \LEMBoxEpimorphisms{} does not hold for surjective $\PARSIMONIOUSBOX$-morphisms.

\begin{eg}
  The monotone surjection $\phi:\boxobj{3}\ra\boxobj{2}$ defined by
  $$\phi(x,y,\half\pm\half)=(\coconnect_{\pm}(x,y),\half\pm\half)$$
  does not preserves intervals.  
\end{eg}

We give a couple of characterizations of $\PARSIMONIOUSBOX$.  
The starting point is the following lifting property that $\PARSIMONIOUSBOX$ shares with most if not all variants of $\BOX$ in the literature.

\begin{lem}
  \label{lem:split.generalized.coconnections}
  Consider the solid diagram in $\PARSIMONIOUSBOX$ of the form
  \begin{equation*}
    \begin{tikzcd}
      & \boxobj{m}\ar[d,two heads]\\
      {[1]}\ar[ur,dotted]\ar[r,hookrightarrow] & \boxobj{n}
    \end{tikzcd}
  \end{equation*}
  with the right vertical arrow surjective and the bottom horizontal arrow injective.  
  Then there exists a dotted $\PARSIMONIOUSBOX$-morphism making the entire diagram commute.  
\end{lem}
\begin{proof}
  Let $\phi$ be the right vertical arrow. 
  Let $\delta$ be the bottom horizontal arrow.  
  
  There exists an element $\delta^*(0)$ in the preimage of $\delta(0)$ under $\phi$ that is maximal in $\boxobj{m}$ among all such elements by $\boxobj{m}$ finite and $\phi$ surjective.  
  Let $I$ be the interval in $\boxobj{m}$ with $\min\,I=\delta^*(0)$ and $\max\,I=\max\,\boxobj{m}$.  
  Then $\phi(I)$ contains $\delta^*(0)$ as its minimum and $\max\,\boxobj{n}$ as its maximum because monotone surjections are extrema-preserving.  
  And $\phi(I)$ is an interval in $\boxobj{n}$ by $\phi$ interval-preserving.  

  Therefore $\phi(I)$ contains $\delta(1)$.  
  Therefore there exists an element $\delta^*(1)$ in $I$ in the preimage of $\delta(1)$ under $\phi$ that is minimal in $I$ among all such elements by $I$ finite and $\delta(1)\in\phi(I)$.  
  Then $\delta^*(1)$ is an immediate successor to $\delta^*(0)$ by maximality of $\delta^*(0)$ and minimality of $\delta^*(1)$.
  Hence the function $\delta^*:[0]\ra\boxobj{m}$ is monotone by $\delta^*(0)=\min\,I\leqslant\delta^*(1)$ by $\delta^*(1)\in I$ and interval-preserving by $\delta^*(1)$ an immediate successor to $\delta^*(0)$.  
  Hence $\delta^*$ defines the desired dotted $\PARSIMONIOUSBOX$-morphism. 
\end{proof}

The following splitting gives symmetric monoidal generators for $\PARSIMONIOUSBOX$. 

\begin{lem}
  \label{lem:monotone.surjections}
  For each solid surjective $\PARSIMONIOUSBOX$-morphism as given in the diagram
  \begin{equation*}
    \begin{tikzcd}
	    \boxobj{m+n+1}\ar[rr,two heads]\ar[dotted]{dr}[description]{\cong} & & \boxobj{n+1}\\
	    & \boxobj{m+n+1}\ar[dotted]{ur}[description]{\phi_1\otimes\phi_2\cdots\phi_{n+1}}
    \end{tikzcd}
  \end{equation*}
  there exist monotone functions $\phi_1,\phi_2,\ldots,\phi_{k+1}$ from $\PARSIMONIOUSBOX$-objects to $[1]$ and dotted left diagonal coordinate permutation making the entire diagram commute. 
\end{lem}
\begin{proof}
  It suffices to take $n>0$.
  Define projection $\sigma_{(s_1,\ldots,s_k);j}:\boxobj{j}\ra\boxobj{k}$ by 
  $$\sigma_{(s_1,\ldots,s_k);j}(x_1,\ldots,x_{j})=(x_{s_1},x_{s_2},\ldots,x_{s_k}),$$
  for each $j$ and subsequence $(s_1,\ldots,s_k)$ of $(1,2,\ldots,j)$. 
  Let $\phi$ denote the top horizontal arrow. 
  
  For each $1\leqslant i\leqslant n+1$, there exists a unique minimal subsequence $s(i)$ of $(1,2,\ldots,m+1)$, non-empty by $\phi$ surjective, and monotone function $\phi_i$ from a $\PARSIMONIOUSBOX$-object to $[1]$, unique by minimality, such that $\sigma_{(i);n+1}\phi=\phi_i\sigma_{s(i);m+n+1}$.
  
  It suffices to show that for distinct $i$ and $j$, $s(i)$ and $s(j)$ have no common integer. 
  Hence it suffices to take the case $n=1$ and show that $s(1)$ and $s(2)$ have no common integer.

  Suppose, to the contrary, that $s(1)$ and $s(2)$ have a common integer.
  We can take that common integer to be $1$ without loss of generality by composing $\phi$ with a coordinate transposition if necessary.  
	Then there exist $x,y\in\boxobj{m+1}$ such that $\sigma_{2;2}\phi([1]\otimes\{x\})=[1]$ and $\sigma_{1;2}\phi([1]\otimes\{y\})=[1]$. 
  We can make our choices so that the number of coordinates in which $x$ and $y$ differ is the minimum possible natural number $k$.   
  Let $K$ be the smallest Boolean interval in $\boxobj{m+1}$ containing $x$ and $y$; thus the dimension of $K$ is $k$.   
	There exist $\bar{x},\bar{y}\in[1]$ such that $\phi([1]\otimes\{x\})=[1]\otimes\{\bar{x}\}$ and $\phi([1]\otimes\{y\})=\{\bar{y}\}\otimes[1]$ by $\phi$ interval-preserving. 
  
  Consider the case $k=0$, 
  Then $[1]\otimes\{\bar{x}\}=\{\bar{y}\}\otimes[1]$, a contradiction. 

  Consider the case $k=1$.  
  Then $[1]\otimes\{x,y\}$ would be a $2$-dimensional Boolean interval whose image under $\phi$ is $([1]\otimes\{\bar{x}\})\cup(\{\bar{y}\}\otimes[1])$, not a Boolean interval in $\boxobj{2}$, a contradiction.  

  Consider the case $k\geqslant 2$ and $\bar{x}=0$.   
  Let $J$ be the smallest interval in $\boxobj{m+1}$ containing $x$ and $\max\,K$ as its extrema.
  Then $\phi([1]\otimes J)$ contains both $\phi(0,x)=(0,\bar{x})=(0,0)$ and $\phi(1,\max\,J)=\phi(\max\,\boxobj{m+n+1})=(1,1)$ by $\phi$ a monotone surjection and hence extrema-preserving.  
  Therefore the restriction of $\phi$ to $[1]\otimes J$ is surjective by interval-preservation.
  For each $x'\in J\setminus\{x\}$, $\phi([1]\otimes\{x'\})$ is a singleton because otherwise we could have replaced $x$ or $y$ with $x'$, contradicting the minimality of $k$.  
  Then for each $x'\in J\setminus\{x\}$, $\phi([1]\otimes x')$ must be a singleton consisting of a point in $\boxobj{2}$ whose first coordinate is $1$ by monotonicity, contradicting surjectivity.  
  
  Dually we get a contradiction for the case  $k\geqslant 2$ and $\bar{x}=1$.  
\end{proof}

\begin{eg}
  \label{eg:3.2.surjection}
  For each solid surjective $\PARSIMONIOUSBOX$-morphism as given in the diagram
  \begin{equation*}
    \begin{tikzcd}
	    \boxobj{3}\ar[rr,two heads]\ar[dotted]{dr}[description]{\cong} & & \boxobj{2}\\
	    & \boxobj{3}\ar[ur,dotted]
    \end{tikzcd}
  \end{equation*}
  there exists a left diagonal coordinate permutation and right diagonal codegeneracy or coconnection making the entire diagram commute.
\end{eg}

A simple consequence is that the surjective $\PARSIMONIOUSBOX$-morphisms are the split epis in $\PARSIMONIOUSBOX$.  

\begin{lem}
  \label{lem:split.epis}
  Every surjective $\PARSIMONIOUSBOX$-morphism splits in $\PARSIMONIOUSBOX$.
\end{lem}
\begin{proof}
  Consider a surjective  $\PARSIMONIOUSBOX$-morphism $\phi:\boxobj{m}\ra\boxobj{n}$.  
  In the case $n=0$, then $\phi$ admits as a section any function $[0]\ra\boxobj{m}$, an interval-preserving monotone function, and hence $\phi$ is split. 
  In the case $n\geqslant 1$, $\phi$ is a tensor product of split epis [Lemmas \ref{lem:split.generalized.coconnections} and \ref{lem:monotone.surjections}] and hence is split.  
 \end{proof}

Recall that a \textit{generalized Reedy category} is a small category $\shape{1}$ equipped with a pair of wide subcategories $\shape{1}^{\mins}$ and $\shape{1}^{\pls}$ and function $\mathrm{deg}$ from the object set of $\shape{1}$ to the set $\N$ of natural numbers such that the following all hold:
\begin{enumerate}
  \item \ldots $\shape{1}^{\mins}\cap\shape{1}^{\pls}$ is the core of $\shape{1}$
  \item \ldots for each $\shape{1}^{\mins}$-morphism $o_1\ra o_2$, $\deg(o_1)\geqslant\deg(o_2)$
  \item \ldots for each $\shape{1}^{\pls}$-morphism $o_1\ra o_2$, $\deg(o_1)\leqslant\deg(o_2)$
  \item \ldots a morphism $o_1\ra o_2$ in $\shape{1}^{\mins}$ or $\shape{1}^{\pls}$ is in both if $\deg(o_1)=\deg(o_2)$ 
  \item \ldots for each $\shape{1}_{\mins}$-morphism $\zeta$, a $\shape{1}$-isomorphism $\gamma$ is the identity if $\gamma\zeta=\zeta$.
  \item \ldots each $\shape{1}$-morphism factors as a $\shape{1}_{\mins}$-morphism followed by a $\shape{1}_{\pls}$-morphism uniquely up to $\shape{1}$-isomorphism
\end{enumerate}
\vspace{.1in}

In practice, the generalized Reedy categories in this paper are small subcategories $\shape{1}$ of $\SETS$ for which $\shape{1}^{\pm}=\shape{1}_{\pm}$. 
We henceforth regard $\PARSIMONIOUSBOX$ as a generalized Reedy category by the following proposition.

\begin{prop}
  \label{prop:generalized.reedy}
  The category $\PARSIMONIOUSBOX$ is a generalized Reedy such that $$\PARSIMONIOUSBOX^{\mins}=\PARSIMONIOUSBOX_{\mins},\quad\PARSIMONIOUSBOX^{\pls}=\PARSIMONIOUSBOX_{\pls},\quad\mathrm{deg}\,\boxobj{n}=n,\,n=0,1,2,\ldots.$$ 
\end{prop}
\begin{proof}
	The $(\PARSIMONIOUSBOX_{\mins}\cap\PARSIMONIOUSBOX_{\pls})$-morphisms are bijections and hence $\PARSIMONIOUSBOX$-isomorphisms [\LEMBoxAutomorphisms].  
  For each injection $\boxobj{m}\ra\boxobj{n}$, $2^m\leqslant 2^n$ and hence $\deg\boxobj{m}=m\geqslant n=\deg\boxobj{n}$.  
  For each surjection $\boxobj{m}\ra\boxobj{n}$, $2^m\geqslant 2^n$ and hence $\deg\boxobj{m}=m\geqslant n=\deg\boxobj{n}$.  
  Each $\PARSIMONIOUSBOX$-morphism $\phi:\boxobj{m}\ra\boxobj{n}$ factors into its corestriction onto its image followed by the inclusion of an interval into $\boxobj{n}$.  
	Therefore $\PARSIMONIOUSBOX$-morphism factors into a $\PARSIMONIOUSBOX_{\mins}$-morphism followed by a $\PARSIMONIOUSBOX_{\pls}$-morphism [\LEMBoxInclusions].  
  Each $\PARSIMONIOUSBOX_{\mins}$-morphism $\zeta:\boxobj{m}\ra\boxobj{n}$ is epi [Lemma \ref{lem:split.epis}] and hence $\gamma\zeta=\id_{\boxobj{n}}\zeta$ implies $\gamma=\id_{\boxobj{n}}$.  
\end{proof}

Recall that an \textit{Eilenberg-Zilber (EZ) category} is a small category $\shape{1}$ equipped with a function $\mathrm{deg}$ from the object set of $\shape{1}$ to the set $\N$ of natural numbers such that \ldots
\begin{enumerate}
  \item \ldots $\deg(o_1)=\deg(o_2)$ if $o_1$ and $o_2$ are $\shape{1}$-isomorphic
  \item \ldots $\deg(o_1)>\deg(o_2)$ if there exists a non-invertible split epi $o_1\ra o_2$ in $\shape{1}$
  \item \ldots $\deg(o_1)<\deg(o_2)$ if there exists a non-invertible mono $o_1\ra o_2$ in $\shape{1}$
  \item \ldots each $\shape{1}$-morphism factors as a split epi followed by a mono
  \item \ldots any pair of split epis has an absolute pushout in $\shape{1}$
\end{enumerate}
where an \textit{absolute pushout} in a category $\cat{1}$ is a pushout in $\cat{1}$ whose image under every functor $\cat{1}\ra\cat{2}$ is a pushout in $\cat{2}$.

\begin{prop}
  \label{prop:eilenberg.zilber}
  The category $\PARSIMONIOUSBOX$ is Eilenberg-Zilber.
\end{prop}

The proof uses a recent observation that a generalized Reedy category $\shape{1}$ is Eilenberg-Zilber if each $\shape{1}^{\mins}$-morphism is a split epi and each $\shape{1}^{\mins}$-morphism $\zeta:o_1\ra o_2$ is uniquely determined up to isomorphism as an $(o_1/\shape{1})$-object by the set of all $\shape{1}$-morphisms $\iota$ with codomain $o_1$ such that $\zeta\iota$ is a $\shape{1}$-isomorphism \cite[Theorem 5.6]{campion2023cubical}.

\begin{proof}
  The category $\PARSIMONIOUSBOX$ is a generalized Reedy category [Proposition \ref{prop:generalized.reedy}] in which $\PARSIMONIOUSBOX^{\pls}=\PARSIMONIOUSBOX_{\pls}$ and $\PARSIMONIOUSBOX^{\mins}=\PARSIMONIOUSBOX_{\mins}$, so that in particular $\PARSIMONIOUSBOX^{\mins}$-morphism is split [Lemma \ref{lem:split.epis}].  

  Consider $\PARSIMONIOUSBOX_{\mins}$-morphism, a surjective $\PARSIMONIOUSBOX$-morphism $\phi:\boxobj{m}\ra\boxobj{n}$.
  It suffices to show that $\phi$ is uniquely determined up to coordinate permutation on $\boxobj{n}$ by the set of all sections to $\phi$ up to coordinate permutation \cite[Theorem 5.6]{campion2023cubical}.
  
  In the case $n=0$, $\boxobj{n}$ is terminal and $\phi:\boxobj{m}\ra[0]$ is the unique $\PARSIMONIOUSBOX$-morphism of the form $\boxobj{m}\ra\boxobj{n}$.  

  Consider the case $n>0$. 
  There exist surjective $\PARSIMONIOUSBOX$-morphisms $\phi_1,\phi_2,\ldots,\phi_n$ each of whose codomains is $[1]$ such that $\phi=\phi_1\otimes\cdots\otimes\phi_n$ [Lemma \ref{lem:monotone.surjections}].  
  Each section $\delta$ to $\phi$ up to coordinate permutation is a composite of cofaces [\LEMBoxInclusions] and therefore decomposes into a tensor product $\delta=\delta_1\otimes\cdots\otimes\delta_n$.  
  It therefore suffices to consider the case $n=1$.  
	In that case, each section to $\phi$ is uniquely determined by its image, a $1$-dimensional interval in $\boxobj{m}$ or equivalently the data of a minimal pair $(x,y)\in\OP{(\boxobj{m})}\times\boxobj{m}$ with $x\leqslant_{\boxobj{n}}y$ and $\phi(x)=0<1=\phi(y)$.  
  The set $S_\phi$ of all such pairs uniquely determines $\phi$ because $\phi(x)=0$ if and only if there exists $(x^*,y^*)\in S_\phi$ with $x\leqslant_{\boxobj{m}}x^*$ and $\phi(y)=1$ if and only if there exists $(x^*,y^*)\in S_\phi$ with $y^*\leqslant_{\boxobj{m}}y$.   
\end{proof}

\begin{prop}
  \label{prop:geometric.characterization}
  Consider a subcategory $\shape{1}$ of $\SETS$ such that the following all hold:
           \begin{enumerate}
		   \item $\shape{1}$ excludes the reversal $[1]\ra[1]$ and diagonal $[1]\ra[1]^2$
		   \item $\shape{1}$ contains $\DEL_1^*$ as a wide subcategory
		   \item every $\shape{1}$-morphism factors into a composite of a surjective $\shape{1}$-morphism followed by an injective $\shape{1}$-morphism.
  \end{enumerate}
  Then $\shape{1}$ is a subcategory of $\PARSIMONIOUSBOX$.
  The choice $\shape{1}=\PARSIMONIOUSBOX$ satisfies all three conditions above.  
\end{prop}
\begin{proof}
 Consider an EZ category $\shape{1}$ fitting into a chain of subcategories
 $$\DEL_1^*\subset\shape{1}\subset\SETS,$$
 with $\DEL_1^*$ wide in $\shape{1}$, and excluding the reversal $[1]\ra[1]$ and diagonal $[1]\ra[1]^2$.		
  Then $\shape{1}$ can be regarded as a submonoidal subcategory of $\POSETS$ [Lemma \ref{lem:reversals}] whose morphisms map $1$-dimensional intervals onto intervals of dimensions $0$ and $1$ [Lemma \ref{lem:diagonals}].

  Consider a $\shape{1}$-morphism $\phi$, necessarily of the form $\boxobj{m}\ra\boxobj{n}$ by $\DEL_1^*$ wide in $\shape{1}$.  
  Fix an interval $I$ in $\boxobj{m}$, the image of a unique composite $\delta_I$ of cofaces [\LEMBoxInclusions], a $\shape{1}$-morphism by $\DEL_1^*$ a subcategory of $\shape{1}$.
  The composite $\phi\delta_I$ in $\shape{1}$ factors as a composite $\phi_{+I}\phi_{-I}$ with $\phi_{-I}$ surjective and $\phi_{+I}$ injective by (3).
  The monotone injection $\phi_{+I}$ is of the form $\boxobj{n_I}\ra\boxobj{n}$ by $\DEL_1^*$ wide in $\shape{1}$.
  Then $\phi(I)$, the image of $\phi_{+I}$, contains exactly $n_I$ distinct maximal chains each of length $n_I$ because $\phi_{+I}$ is a monotone injection mapping $1$-dimensional intervals onto intervals, $1$-dimensional by injectivity. 
  Thus $\phi(I)$ is an interval in $\boxobj{n}$.
  Thus $\phi$ is a $\PARSIMONIOUSBOX$-morphism.

  Therefore $\shape{1}$ is a subcategory of $\PARSIMONIOUSBOX$.  
  
  Additionally, $\PARSIMONIOUSBOX$ is EZ [Proposition \ref{prop:eilenberg.zilber}] and excludes the reversal $[1]\ra[1]$ [Lemma \ref{lem:reversals}] and diagonal $[1]\ra[1]^2$ [Lemma \ref{lem:diagonals}]. 
\end{proof}

\begin{eg}
  The $\POSETS$-product morphism
  $$\gamma_{\mins}\times_{\POSETS}\gamma_{\pls}:\boxobj{2}\ra\boxobj{2}$$
  does not factor into a surjective monotone function between finite Boolean lattices followed by an injective monotone function between finite Boolean lattices.  
	This function is not a $\PARSIMONIOUSBOX$-morphism because its image is a set containing $3$ elements and hence is not an interval in a Boolean lattice.  
\end{eg}

We can decompose $\PARSIMONIOUSBOX$ as follows.

\begin{thm}
  \label{thm:algebraic.characterization}
  The category $\PARSIMONIOUSBOX$ is generated as a symmetric monoidal category by \ldots
  \begin{enumerate}
    \item $\sigma:[1]\ra[0]$
    \item $\delta_{\mins}:[0]\ra[1]$ and $\delta_{\pls}:[0]\ra[1]$
    \item all monotone surjections $\boxobj{n}\twoheadrightarrow[1]$
  \end{enumerate}
  For each $\PARSIMONIOUSBOX$-morphism $\phi:\boxobj{m}\ra\boxobj{n}$, there exists a unique $n$-tuple $(m_1,m_2,\ldots,m_n)$ of natural numbers with $m_1+m_2+\cdots+m_n\leqslant m$ and unique coset in $\Sigma_m/(\Sigma_{m_1}\times\cdots\times\Sigma_{m_n})$ such that for each representative $g\in\Sigma_m$ of that coset, there exist unique monotone functions $\phi_i:\boxobj{m_i}\ra[1]$ for $1\leqslant i\leqslant n$, each non-constant in each of its coordinates, with $\phi g=\phi_1\otimes\cdots\otimes\phi_n\otimes\sigma^{m-m_1-\cdots-m_n}$.
\end{thm}
\begin{proof}
  For each $j$ and subsequence $(s_1,\ldots,s_k)$ of $(1,2,\ldots,j)$, define
  $$\sigma_{(s_1,\ldots,s_k);j}:\boxobj{j}\ra\boxobj{k},\quad 
    \sigma_{(s_1,\ldots,s_k);j}(x_1,\ldots,x_{j})=(x_{s_1},x_{s_2},\ldots,x_{s_k})$$  
  
  Every terminal non-identity $\PARSIMONIOUSBOX$-morphism is a composite of codegeneracies and hence a composite of tensor products of identities with $\sigma$.  
  Every $\PARSIMONIOUSBOX$-morphism factors into a composite of surjective $\PARSIMONIOUSBOX$-morphisms followed by injective $\PARSIMONIOUSBOX$-morphisms [Proposition \ref{prop:generalized.reedy}].  
  Every non-identity injective $\PARSIMONIOUSBOX$-morphism, up to permutation of tensor products, is a composite of cofaces [\LEMBoxInclusions] and hence a tensor  product in $\PARSIMONIOUSBOX$ of monotone functions of the form $[0]\ra[1]$.
  Therefore every non-terminal $\PARSIMONIOUSBOX$-morphism, up to permutation of tensor products, is a tensor product in $\PARSIMONIOUSBOX$ of monotone functions of the form $\boxobj{n}\ra[1]$ [Lemma \ref{lem:monotone.surjections}]. 
  Thus $\PARSIMONIOUSBOX$ is generated as a symmetric monoidal category by (1)-(3).  
  
  Consider a $\PARSIMONIOUSBOX$-morphism $\phi:\boxobj{m}\ra\boxobj{n}$. 
  For each $1\leqslant i\leqslant n$, there exists a unique minimal subsequence $s(i)$ of $(1,2,\ldots,m)$ such that $\sigma_{(i);n}\phi$ factors through $\sigma_{s(i)}$. 
  Let $m_0=0$.  
  Let $m_i$ be the length of $s(i)$.  
  Let $t(i)$ be the subsequence $(m_1+\cdots+m_{i-1}+1,\ldots,m_1+\cdots+m_i)$ of $(1,2,\ldots,m)$.  
  
  The sequences $s(1),s(2),\ldots,s(n)$ are disjoint because $\phi$ is a tensor product of monotone functions to $[1]$ up to coordinate permutations.  
  Therefore $m_1+m_2+\cdots+m_n\leqslant m$.  
  There exists a unique coset in $\Sigma_m/(\Sigma_{m_1}\times\cdots\times\Sigma_{m_n})$ whose representatives map the elements in $t(i)$ onto the elements of $s(i)$ for each $1\leqslant i\leqslant n$.  

  Fix such a representative $g\in\Sigma_m$.  
  Define $\phi_i:\boxobj{m_i}\ra[1]$ by $\phi_i\sigma_{t(i);m}=\sigma_{(i);n}\phi g$.  
  Then $\phi_i$ is non-constant by minimality of $m_i$.  
  Hence $\phi g$ is constant on its last $m-m_1-m_2-\cdots-m_n$ coordinates. 
  Thus $\phi g=\phi_1\otimes\cdots\phi_{m_n}\otimes\sigma^{m-m_1-\cdots-m_n}$.
\end{proof}

Take an \textit{operad} to simply mean a symmetric operad on the Cartesian mononidal category $\SETS$ \cite{may2006geometry}. 
Just as operads determine monads, monads on $\SETS$ determine operads \cite[Exercise 2.2.11]{leinster2004higher}.  
We detail one such example.
Let $\mathscr{O}$ be what we term the \textit{distributive lattice operad}, the operad $\mathscr{O}$ such that $\mathscr{O}(n)$ underlies the free distributive lattice on $\{1,2,\ldots,n\}$ equipped with the action of $\Sigma_n$ permuting the generators $1,2,\ldots,n$ such that the operad actions are defined by commutative diagrams 
\begin{equation*}
	\begin{tikzcd}
		\mathscr{O}(k)\times\prod_{i=1}^k\mathscr{O}(n_i)\ar[rr,dotted]\ar[d,hookrightarrow] & & \mathscr{O}(n_1+\cdots+n_k)
		\\
		\mathscr{O}(k)\times\mathscr{O}(n_1+\cdots+n_k)^{k}\ar[rr,equals] & & \mathscr{O}(k)\times\DISLATS(\mathscr{O}(k),\mathscr{O}(n_1+\cdots+n_k))\ar[u]
	\end{tikzcd}
\end{equation*}
in which the left vertical arrow is induced by injections $\{1,2,\ldots,n_i\}\ra\{1,2,\ldots,n_1+\cdots+n_k\}$ sending each $j$ to $(n_1+\cdots+n_{i-1})+j$, the bottom horizontal identification is induced from the identification $\SETS(\{1,2,\ldots,k\},\mathscr{O}(n_1+\cdots+n_k))=\DISLATS(\mathscr{O}(k),\mathscr{O}(n_1+\cdots+n_k))$, and the right vertical arrow is defined by evaluation.  
We can make the natural identification
$$\mathscr{O}(n)=\PARSIMONIOUSBOX(\boxobj{n},[1]),$$
where the right side is equipped with the lattice operations defined element-wise from $[1]$.  
In other words, the endomorphism operad of $[1]$ in $\PARSIMONIOUSBOX$ is $\mathscr{O}$.  
Under this identification, we can regard a monotone function $\boxobj{n}\ra[1]$ as an analogue of an $n$-variable polynomial, where binary suprema and infima play the of addition and multiplication; this abstract $n$-variable polynomial is exactly a term in the free distributive lattice on $n$-generators $x_1,x_2,\ldots,x_n$.  
This identification allows us to reinterpret a composite $\phi(\phi_1,\phi_2,\ldots,\phi_k)$ of monotone functions $\phi:\boxobj{k}\ra[1]$ and $\phi_i:\boxobj{n_i}\ra[1]$ for $1\leqslant i\leqslant k$ as the term in $\mathscr{O}(n_1+\cdots+n_k)$ obtained by substituting the generators in the term $\phi\in\mathscr{O}(k)$ with the terms $\phi_i\in\mathcal{O}(n_i)$.  
The theorem implies that the symmetric submonoidal category of $\PARSIMONIOUSBOX$ containing all but the non-terminal non-identity morphisms is generated, albeit not freely, by the distributive lattice operad. 

\begin{rem}
  All the monotone surjections of the form
  $$[1]\ra[0],\quad\boxobj{n}\twoheadrightarrow[1]$$
  can be regarded as \textit{generalized codegeneracies} in $\PARSIMONIOUSBOX$ that include the original codegeneracies, coconnections of both kinds, and many more kinds of maps.  
  The category $\PARSIMONIOUSBOX$ is generated by the cofaces and the generalized codegeneracies.
\end{rem}

The $\DEL_1[\tau,\gamma_{\mins}]^*$-morphisms have a simple explicit characterization.

\begin{cor}
  \label{cor:join.characterization}
  The following are equivalent for a function of the form
  $$\phi:\boxobj{m}\ra\boxobj{n}.$$
  \begin{enumerate}
    \item $\phi$ is a $\DEL_1[\tau,\gamma_{\mins}]^*$-morphism
    \item $\phi$ defines an interval-preserving meet-semilattice homomorphism
  \end{enumerate}
\end{cor}
\begin{proof}
  Each $\DEL_1[\tau,\gamma_{\mins}]^*$-morphism is an interval-preserving meet-semilattice homomorphism because tensor products preserve such maps and $\tau,\gamma_{\mins}$, and identities on finite Boolean lattices all examples of such  maps.   
  Consider an interval-preserving meet semilattice homomorphism $\phi$ of the form $\boxobj{m}\ra\boxobj{n}$.  

  It suffices to show $\phi$ is a $\DEL_1[\tau,\gamma_{\mins}]^*$-morphism.
  The function $\phi$ factors as a composite of its corestriction to its image in $\boxobj{n}$, a surjective interval-preserving meet semilattice homomorphism, followed by an inclusion of intervals, a $\DEL_1^*[\tau]$-morphism [\LEMBoxInclusions].
  Therefore it suffices to take the case $\phi$ a surjective interval-preserving meet semilattice homomorphism.  
  A tensor product of $\PARSIMONIOUSBOX$-morphisms $\alpha,\beta$ is a meet semilattice homomorphism if and only if $\alpha,\beta$ are both meet semilattice homomorphisms.  
  Thus it suffices to take the case $n=1$ [Theorem \ref{thm:algebraic.characterization}]. 
  There exists a unique element $p_\phi$ in the free distributive lattice on $m$ generators $t_1,\ldots,t_m$ such that $\phi(x_1,\ldots,x_n)$ corresponds to the evaluation of $p_\phi$ as an element in $[1]$ when $t_i=x_i$ for each $(x_1,\ldots,x_m)\in\boxobj{m}$ by Birkhoff Duality \cite{church1940numerical}. 
  Then there exists some subset $I\subset\{1,2,\ldots,m\}$ such that $\phi(x)=\bigwedge_{i\in I}x_i$ because otherwise $\phi$ would not preserve binary meets.
  Hence $\phi$ is a composite of codegeneracies and coconnections of the first kind.  
\end{proof}

The $\DEL_1[\tau,\gamma_{\pls}]^*$-morphisms have a dual characterization.

\begin{cor}
  \label{cor:meet.characterization}
  The following are equivalent for a function of the form
  $$\phi:\boxobj{m}\ra\boxobj{n}.$$
  \begin{enumerate}
    \item $\phi$ is a $\DEL_1[\tau,\gamma_{\pls}]^*$-morphism
    \item $\phi$ is an interval-preserving join-semilattice homomorphism
  \end{enumerate}
\end{cor}

The $\DEL_1[\tau]^*$-morphisms are already known to have a simple characterization.

\begin{thm:minimal.symmetric.monoidal.site}
  The following are equivalent for a function of the form
  $$\phi:\boxobj{m}\ra\boxobj{n}.$$
  \begin{enumerate}
    \item $\phi$ is a  $\DEL_1[\tau]^*$-morphism
    \item $\phi$ is an interval-preserving lattice homomorphism
  \end{enumerate}
\end{thm:minimal.symmetric.monoidal.site}

\section{Presheaves}\label{sec:presheaves}
Presheaves over general small categories are interpretable as combinatorial models of spaces.  
Presheaves over small subcategories of $\CATS$ are interpretable as combinatorial models of directed spaces.  
For each small subcategory $\shape{1}$ of $\CATS$, let $\nerve_{\shape{1}}$ denote the \textit{nerve functor} 
$$\nerve_{\shape{1}}=(\shape{1}\ira\CATS)^*:\CATS\ra\hat{\shape{1}}.$$

The Yoneda Lemma implies that we can make the following identification
$$\hat{\shape{1}}(\shape{1}[o],C)=C(o)$$ 
natural in all small categories $\shape{1}$, $\shape{1}$-objects $o$, and $\hat{\shape{1}}$-objects $C$. 
A \textit{vertex} of a $\SETS$-valued presheaf $C$ over an EZ category $\shape{1}$ is an element in $C(o)$ with $\deg(o)=0$. 
A $\SETS$-valued presheaf over a small category is \textit{atomic} if it is the quotient of a representable.  
The \textit{dimension} of a presheaf $C$ over an EZ category $\shape{1}$ is $-1$ if $C=\varnothing$ and otherwise is the infimum over all natural numbers $n$ such that $C$ is a colimit of representables $\shape{1}[o]$ with $\catfont{deg}(o)\leqslant n$.  
For each EZ category $\shape{1}$ and $\shape{1}$-object $o$, let $\partial\shape{1}[o]$ denote the maximal subpresheaf of the representable $\shape{1}[o]$ having dimension strictly less than $\catfont{deg}(o)$. 
For each $\SETS$-valued presheaf $C$ over a small category $\shape{1}$ and a $(\shape{1}/C)$-object $\theta$, write $\Star_C(\theta)$ for the union of all atomic subpresheaves $S\subset C$ to which $\theta$ corestricts; call $\Star_C(\theta)$ the \textit{closed star} of $\theta$ in $C$.

\subsection{Simplicial}
We call the objects and morphisms of $\SIMPLICIALSETS$ respectively \textit{simplicial sets} and \textit{simplicial functions}. 
Write $\sd_{k+1}$ for the continuous, cocontinuous endofunctor
$$\sd_{k+1}=\SETS^{\OP{((-)^{\oplus(k+1):\DEL\ra\DEL})}}:\SIMPLICIALSETS\ra\SIMPLICIALSETS.$$

The simplicial set $\sd_{2}S$ is referred to elsewhere as the \textit{ordinal subdivision} of a simplicial set $S$ \cite{ehlers2008ordinal}.  
Intuitively, $\sd_{k+1}S$ encodes a $(k+1)$-fold edgewise subdivision of the simplicial set $S$.  

\begin{eg}
  We can make the natural identification $\sd_{1}=\id_{\SIMPLICIALSETS}$.

\end{eg}

Write $\epsilon$ for the natural transformation $\sd_3\ra\id_{\SIMPLICIALSETS}$ induced by the monotone function $0\mapsto 1:[0]\ra[0]\oplus[0]\oplus[0]$.

\subsection{Cubical}
Henceforth fix a small monoidal category $\BOX$ containing $\DEL_1^*$ as a wide submonoidal subcategory.
We call the objects and morphisms of $\hat\BOX$ \textit{cubical sets} and \textit{cubical functions}. 
For each cubical set $C$, let $C_n=C(\boxobj{n})$.  
For a given cubical set $C$, define
$$d_{\pm i;n+1}=C(\delta_{\pm i;n+1}):C_{n+1}\ra C_n,\quad s_{i;n+1}:C(\sigma_{i;n+1}):C_n\ra C_{n+1}$$

Day convolution extends the tensor product on $\BOX$, a restriction and corestriction of the Cartesian monoidal product on $\POSETS$, to a cococontinuous tensor product allowing us to regard $\hat\BOX$ as a monoidal category with unit $\BOX[0]$.  

\subsubsection{Boolean complexes}\label{sec:boolean.complexes}
Suppose $\BOX$ lies in a chain of categories
$$\DEL_1[\tau]^*\subset\BOX\subset\POSETS$$
with $\DEL_1[\tau]^*$ wide in $\BOX$.  
For each poset $P$, define $\BOX[P]$ to be the subpresheaf of $\cnerve\,P$ whose cubes are those monotone functions $\boxobj{n}\ra P$ corestricting to Boolean intervals $I$ in $P$ such that the corestrictions $\boxobj{n}\ra I$ are $(\boxobj{n}/\POSETS)$-isomorphic to $\BOX$-morphisms. 
In the case that $\BOX$ is a subcategory of $\PARSIMONIOUSBOX$ that contains the coordinate permutations and admits a (surjective, injective) factorization for its morphisms, then the atomic subpresheaves of $\BOX[P]$ correspond to the Boolean intervals in $P$.  

\begin{eg}
  For $P=\boxobj{n}$, $\BOX[P]$ is the usual representable
  $$\BOX\boxobj{n}=\BOX(-,\boxobj{n}):\OP{(\boxobj{n})}\ra\SETS.$$
\end{eg}

\begin{eg}
  For $\BOX=\DEL_1[\tau,\gamma_{\mins},\gamma_{\pls},\diagonal]^*$, $\BOX[P]_n$ is the set of monotone functions
  $$\boxobj{n}\ra P$$
  whose images lie inside Boolean intervals in the poset $P$.
\end{eg}

\begin{eg}
  The subset $\BOX[P]_n\subset(\nerve_\BOX P)_n=\POSETS(\boxobj{n},P)$ consists of those functions
  $$\boxobj{n}\ra P$$
  mapping Boolean intervals onto Boolean intervals and defining \ldots
  \begin{enumerate}
    \item \ldots monotone functions if $\BOX=\PARSIMONIOUSBOX$
    \item \ldots join-semilattice homomorphisms if $\BOX=\DEL_1[\tau,\coconnect_{\pls}]^*$
    \item \ldots meet-semilattice homomorphisms if $\BOX=\DEL_1[\tau,\coconnect_{\mins}]^*$
    \item \ldots lattice homomorphisms if $\BOX=\DEL_1[\tau]^*$
  \end{enumerate}
\end{eg}

We can abstractly identify when a monotone function $P_1\ra P_2$ of posets induces a cubical function $\BOX[P_1]\ra\BOX[P_2]$.

\begin{prop}
  \label{prop:boolean.functions}
  Suppose $\BOX$ lies in a chain of subcategories
  $$\DEL_1[\transpose]^*\subset\BOX\subset\POSETS$$
  with $\DEL_1[\transpose]^*$ wide in $\BOX$.  
  The following are equivalent for a monotone function $\phi:P_1\ra P_2$ of posets:
  \begin{enumerate}
	  \item There exists a unique dotted cubical function making the following commute:
    \begin{equation*}
    \begin{tikzcd}
	    \BOX[P_1]\ar[d,hookrightarrow]\ar[dotted]{r}[above]{\BOX[\phi]} & \BOX[P_2]\ar[d,hookrightarrow]\\
	    \cnerve P_1\ar{r}[below]{\cnerve\phi} & \cnerve P_2
    \end{tikzcd}
  \end{equation*}
  \item Each restriction of $\phi$ to a finite Boolean interval corestricts to a monotone function $\POSETS^{[1]}$-isomorphic to a $\BOX$-morphism.
	  \end{enumerate}
  \end{prop}
  \begin{proof}
    Suppose (1).
    Consider a finite Boolean interval $I$ in $P_1$.  
    Then $I$ is $\POSETS$-isomorphic to a $\BOX$-object by $\DEL_1[\transpose]^*$ wide in $\BOX$. 
    Therefore there exists a monotone injection of the form $\theta_1:\boxobj{n}\ra P_1$ with image $I$.  
	  Then $\theta_1\in\BOX[P_1]_n$ because the corestriction $\boxobj{n}\ra I$ of $\theta_1$ is $(\boxobj{n}/\POSETS)$-isomorphic to $\id_{\boxobj{n}}$.  
    And $\phi\theta_1$ corestricts to a monotone function $(\boxobj{n}/\POSETS)$-isomorphic to a $\BOX$-morphism $\theta_2$, by (1). 
	  Thus the restriction of $\phi$ to $I$, $\POSETS^{[1]}$-isomorphic to $\phi\theta_1$, corestricts to a monotone function $\POSETS^{[1]}$-isomorphic to $\theta_2$.   
    Thus (2).  

    Suppose (2).  
    Consider $\theta\in\BOX[P_1]_n$.  
    Then there exists a Boolean interval $I_1$ in $P_1$ and corestriction $\boxobj{n}\ra I_1$ of $\theta$ that is $(\boxobj{n}/\POSETS)$-isomorphic to a $\BOX$-morphism $\theta^*$ by definition of $\BOX[P_1]$.   
    There exists a Boolean interval $I_2$ in $P_2$ such that $\phi$ restricts and corestricts to a monotone function $I_1\ra I_2$ $\POSETS^{[1]}$-isomorphic to a $\BOX$-morphism $\phi^*$ by (2).  
   Therefore $\phi\theta$ corestricts to a monotone function $\boxobj{n}\ra I_2$$(\boxobj{n}/\POSETS)$-isomorphic to the $\BOX$-morphism $\phi^*\theta^*$.  
    Hence (1).  
\end{proof}

\begin{eg}
  A function $\phi:P_1\ra P_2$ between posets induces a cubical function 
  $$\BOX[P_1]\ra\BOX[P_2]$$
  precisely when $\phi$ is monotone and maps finite Boolean intervals into Boolean intervals if $\BOX=\DEL_1[\tau,\gamma_{\mins},\gamma_{\pls},\diagonal]^*$.
\end{eg}

\begin{eg}
  A function $\phi:P_1\ra P_2$ between posets induces a cubical function 
  $$\BOX[P_1]\ra\BOX[P_2]$$
  precisely when $\phi$ maps finite Boolean intervals onto Boolean intervals and defines \ldots
  \begin{enumerate}
    \item \ldots a monotone functions if $\BOX=\PARSIMONIOUSBOX$
    \item \ldots a join-semilattice homomorphisms if $\BOX=\DEL_1[\tau,\coconnect_{\pls}]^*$
    \item \ldots a meet-semilattice homomorphisms if $\BOX=\DEL_1[\tau,\coconnect_{\mins}]^*$
    \item \ldots a lattice homomorphisms if $\BOX=\DEL_1[\tau]^*$
  \end{enumerate}
\end{eg}

One of the main advantages of the choices $\BOX=\PARSIMONIOUSBOX,\DEL_1[\tau,\gamma_{\mins},\gamma_{\pls},\diagonal]^*$ is not mathematical but terminological: the functions $\phi:P_1\ra P_2$ between posets that induce cubical functions $\BOX[P_1]\ra\BOX[P_2]$ are concise to state in general precisely when $\BOX=\DEL_1[\tau,\gamma_{\mins},\gamma_{\pls},\diagonal]^*$ or $\BOX=\PARSIMONIOUSBOX$: as those monotone functions which map finite Boolean intervals \textit{into} or \textit{onto} Boolean intervals. 
And among those two choices, only the choice $\BOX=\PARSIMONIOUSBOX$ implies that the induced cubical functions $\BOX[P_1]\ra\BOX[P_2]$ always induce $1$-Lipschitz maps between uniquely geodesic \textit{$\ell_2$-realizations} [Examples \ref{eg:geometric.cubes} and \ref{eg:geometric.realizations}].  

\subsubsection{Subdivision}
Subdivided cubes are special cases of Boolean complexes, cubical sets of Boolean intervals in certain posets $(\boxobj{n})^{[k+1]}$.   
Our first step is to characterize when those cubical sets are in fact natural in $\BOX$-objects $\boxobj{n}$ . 

\begin{cor}
  \label{cor:sd}
  Suppose $\BOX$ is any of the following variants:
  $$\PARSIMONIOUSBOX,\BOX_1[\tau]^*,\BOX_1[\tau,\gamma_{\mins}]^*,\BOX_1[\tau,\gamma_{\pls}]^*,\BOX_1[\tau,\gamma_{\pls},\gamma_{\mins}]^*,\DEL_1[\tau,\gamma_{\mins},\gamma_{\pls},\diagonal]^*$$
  Consider the solid diagrams of the following forms
   \begin{equation*}
    \begin{tikzcd}
	    \BOX[(\boxobj{m})^{[k]}]\ar[d,hookrightarrow]\ar[dotted]{r}[above]{\BOX[\phi^{[k]}]} & \BOX[(\boxobj{n})^{[k]}]\ar[d,hookrightarrow]\\
	    \cnerve(\boxobj{m})^{[k]}\ar{r}[below]{\cnerve\phi^{[k]}} & \cnerve(\boxobj{n})^{[k]}
    \end{tikzcd}
    \quad
    \begin{tikzcd}
	    \BOX[(\boxobj{n})^{[k_2]}]\ar[d,hookrightarrow]\ar[dotted]{r}[above]{\BOX[(\boxobj{n})^{\delta}]} & \BOX[(\boxobj{n})^{[k_1]}]\ar[d,hookrightarrow]\\
	    \cnerve(\boxobj{m})^{[k_2]}\ar{r}[below]{\cnerve(\boxobj{n})^{\delta}} & \cnerve(\boxobj{n})^{[k_1]}
    \end{tikzcd}
  \end{equation*}
	where $\phi$ denotes a $\BOX$-morphism $\phi:\boxobj{m}\ra\boxobj{n}$ and $\delta$ denotes a $\DEL_{\pls}$-morphism.
  For each such solid diagram of either form, there exists a unique dotted cubical function making the entire diagram commute.
\end{cor}

The proof uses the fact, proven elsewhere \cite[Proposition 3.12]{krishnan2023cubicalII}, that 
$$(\boxobj{n})^{\delta}:(\boxobj{n})^{[k_2]}\ra(\boxobj{n})^{k_1}$$ 
is a lattice homomorphism preserving Boolean intervals for each monotone injection $\delta:[k_1]\ra[k_2]$.

\begin{proof}
  Consider a monotone function $\phi:\boxobj{m}\ra\boxobj{n}$. 

  Uniqueness of the dotted cubical functions follows from the monicity of the right vertical cubical functions in each diagram. 

  For each $\DEL_{\pls}$-morphism $\delta:[k_1]\ra[k_2]$, $(\boxobj{n})^{\delta}$ is a lattice homomorphism preserving Boolean intervals [Proposition 3.12, \cite{krishnan2023cubicalII}] and hence its restriction to each Boolean interval corestricts to a monotone function $\POSETS^{[1]}$-isomorphic to a $\DEL_1[\tau]^*$-morphism.  
  Thus there exists a dotted cubical function making the right diagram commute. 

  Consider a monotone function $\phi$ as in the left diagram. 
  Consider a Boolean interval $F$ in $(\boxobj{m})^{[k]}$.
	It suffices to show that the restriction of $\phi^{[k]}$ to $F$ corestricts to a monotone function $\POSETS^{[1]}$-isomorphic to a $\BOX$-morphism if $\phi$ is a $\BOX$-morphism, for $\BOX=\PARSIMONIOUSBOX,\BOX_1[\tau]^*,\BOX_1[\gamma_{\mins},\tau]^*,\BOX_1[\gamma_{\pls},\tau]^*,\BOX_1[\gamma_{\pls},\gamma_{\mins},\tau]^*,\DEL_1[\tau,\gamma_{\mins},\gamma_{\pls},\diagonal]^*$ [Proposition \ref{prop:boolean.functions}].  

  There exists a Boolean interval $I$ in $\boxobj{m}$ and $j\in[k]$ such that $F$ is the set of all monotone functions $[k]\ra I$ sending $0\leqslant i<j$ to $\min\,I$, $j<i\leqslant k$ to $\max\,I$, and $j$ to an element in $I$.  
  Thus all the functions in $\phi^{[k]}(F)$, the composites of all functions in $F$ with $\phi$, is the set of all monotone functions $[k]\ra I$ sending $0\leqslant i<j$ to $\min\,\phi(I)$, $j<i\leqslant k$ to $\max\,\phi(I)$, and $j$ to an element in $\phi(I)$.  
  
  In particular, all the functions in $F$ attain the same value on all but possibly one element in $[k]$.  
  Thus all the functions in $\phi^{[k]}(F)$, the composite of all functions in $F$ with $\phi$, attain the same value on all but possibly one element in $[k]$.   
  Therefore $\phi^{[k]}(F)$ is a subset of a Boolean interval in $(\boxobj{n})^{[k]}$. 
  The case $\BOX=\DEL_1[\tau,\gamma_{\mins},\gamma_{\pls},\diagonal]^*$ follows.

  Suppose $\phi$ is a $\PARSIMONIOUSBOX$-morphism.  
  Then $\phi(I)$ is a Boolean interval in $\boxobj{n}$. 
  Therefore $\phi^{[k]}(F)$ is a Boolean interval in $(\boxobj{n})^{[k]}$ .
  The case $\BOX=\PARSIMONIOUSBOX$ follows.  

  Suppose $\phi$ is a $\DEL_1[\tau,\gamma_{\mins}]^*$-morphism.  
  Then $\phi$ is a join-semilattice homomorphism.  
  Thus $\phi^{[k]}$ is a join-semilattice homomorphism.  
  Hence the restriction of $\phi^{[k]}$ to $F$ corestricts to a join-semilattice homomorphism onto its image because its image $\phi^{[k]}(F)$is  a Boolean interval in $(\boxobj{n})^{[k]}$ and therefore inclusion  $\phi^{[k]}(F)\ira(\boxobj{n})^{[k]}$ is a join-semilattice homomorphism [Corollary \ref{cor:join.characterization}].  
  The case $\BOX=\DEL_1[\tau,\gamma_{\mins}]^*$ follows.
  
  Similarly the case $\BOX=\DEL_1[\tau,\gamma_{\pls}]^*$ follows.

  The cases $\BOX=\DEL_1[\tau,\gamma_{\mins},\gamma_{\pls}]^*$ and $\BOX=\DEL_1[\tau]^*=\DEL_1[\tau,\gamma_{\mins}]^*\cap\DEL_1[\tau,\gamma_{\pls}]^*$ therefore also follow.
\end{proof}

Thus in the case that $\BOX$ is any of the following variants:
$$\PARSIMONIOUSBOX,\BOX_1[\tau]^*,\BOX_1[\tau,\gamma_{\mins}]^*,\BOX_1[\tau,\gamma_{\pls}]^*,\DEL_1[\tau,\gamma_{\mins},\gamma_{\pls}]^*,\DEL_1[\tau,\gamma_{\mins},\gamma_{\pls},\diagonal]^*$$
the construction $\BOX[(-)^{[k]}]$ defines a functor in the diagram
$$\begin{tikzcd}
  \BOX\ar{r}[above]{\BOX[(-)^{[k]}]}\ar{d}[description]{\BOX[-]} & \hat\BOX\\
  \hat\BOX\ar[dotted]{ur}[description]{\sd_{k+1}}
\end{tikzcd}$$
and we can hence take $\sd_{k+1}$ to be the unique left Kan extension of the solid horizontal functor along the Yoneda embedding $\BOX[-]$ in the diagram, let $\epsilon$ denote the unique monoidal natural transformation 
$\epsilon:\sd_3\ra\id_{\CUBICALSETS}$ characterized by the rule 
$$\epsilon_{\BOX[1]}=\BOX[[1]^{0\mapsto 1:[0]\ra[2]}]:\BOX[[1]^{[2]}]\ra\BOX[1].$$

Context will make clear whether $\sd_{k+1}$ refers to the endofunctor on $\SIMPLICIALSETS$ or $\CUBICALSETS$ and thus whether $\epsilon$ refers to the natural transformation to $\id_{\SIMPLICIALSETS}$ or $\id_{\CUBICALSETS}$.  
The comptability between both functors and natural transformations [Lemma \ref{lem:tri}] justifies the abuse in notation to some extent.  

\begin{rem}
  Cubical subdivision can also be constructed when
  $$\BOX=\DEL_1^*,\DEL_1[\gamma_{\mins}]^*,\DEL_1[\gamma_{\pls}]^*,$$
  but under these definitions the subdivision of a cube cannot be described as the Boolean intervals of a poset without introducing extra structure on the poset.  
  However, in the above cases $\BOX$ is the free monoidal category with a prescribed unit generated by a finite subcategory and hence the functoriality of subdivision on representables can be checked by hand.
  In the general case where $\BOX$ has as its objects the finite Boolean lattices but is neither symmetric monoidal nor freely generated by a finite subcategory, the verification that cubical subdivision is functorial on representables is more difficult to verify.  
\end{rem}

A special feature of $\PARSIMONIOUSBOX$ and some subvariants, unlike variants containing the diagonals, are properties of cubical subdivision $\sd_3$ given by the following pair of lemmas.
These properties have been previously given for $\BOX=\DEL_1^*$ \cite{krishnan2015cubical}, $\BOX=\DEL_1^*[\tau]$ \cite{krishnan2023cubicalII}, and $\BOX=\DEL_1^*[\gamma_{\mins}]^*$ \cite{krishnan2022uniform}.
Proofs mimic earlier proofs but are given to illustrate how $\PARSIMONIOUSBOX$ is the largest variant of $\BOX$ for which the properties hold.  

\begin{lem}
  \label{lem:collapse.star}
  Suppose $\BOX$ is one of the following categories:
  $$\PARSIMONIOUSBOX,\DEL_1[\tau]^*,\DEL_1[\gamma_{\mins},\tau]^*,\DEL_1[\gamma_{\pls},\tau]^*,\DEL_1[\gamma_{\mins},\gamma_{\pls},\tau]^*.$$  
  Fix a cubical set $C$.
  The following statements all hold.
  \begin{enumerate}
    \item For each $v\in(\sd_3C)_{0}$, $\epsilon_C(\Star_{\sd_3C})(v)\subset\support_{\sd_3}(v,C).$
    \item For each non-empty subpresheaf $A$ of an atomic subpresheaf of $\sd_3C$, there exist \ldots
      \begin{enumerate}
        \item unique minimal subpresheaf $B\subset C$ with $A\cap\sd_3B\neq\varnothing$; and
	\item unique retraction $\mu:A\ra A\cap\sd_3B$ induced by $\DEL_1[\tau]^*$-morphisms
      \end{enumerate}
      Moreover, $A\cap\sd_3B$ is isomorphic to a representable if $A$ is atomic.
      And $\epsilon_C(A\ira\sd_3C)=(A\cap\sd_3B\ira\sd_3C)\mu$.
  \end{enumerate}
\end{lem}
\begin{proof}
  It suffices to take $C$ atomic by minimality. 

  \vspace{.1in}
  \textit{(1), (2a), (2b) and the last equality:}
  It suffices to take $C$ representable by naturality.  
  Atomic subpresheaves of representables are representables themselves because images of $\BOX$-morphisms are intervals.  
  Therefore $A$ is representable.  

  It further suffices to consider the case $C=\BOX[1]$ and hence $A\cong\BOX[0]$ or $A\cong\BOX[1]$ because closed stars and $\sd$ commute with tensor products.  
  In that case, (1), (2a), (2b) and the last equality in (2) follow from inspection.

  \vspace{.1in}
  \textit{representability of $A\cap\sd_3B$}
  Note $A\cap\sd_3\partial B=\varnothing$ by minimality of $B$.  
  Thus there exists a monic restriction of the quotient cubical function $\sd_3(\BOX\boxobj{\dim\,B}\ra B)$ to a subpresheaf of $\BOX\boxobj{\dim\,B}$ whose image contains $A\cap\sd_3B$.  
  Thus $A\cap\sd_3B$ is isomorphic to a subpresheaf of $\sd_3\BOX\boxobj{\dim\,B}$ and hence is a representable by minimality.
\end{proof}

\begin{lem}
  \label{lem:retractions}
  Consider the solid commutative diagram
	\begin{equation*}
          \begin{tikzcd}
		  A'\ar[r]\ar[d] & A''\ar[d]\\
	    A'\cap\sd_3B'\ar[r,dotted] & A''\cap\sd_3B''
	  \end{tikzcd}
	\end{equation*}
	where $A',A''$ are respectively atomic subpresheaves of $\sd_3C',\sd_3C''$ and $B',B''$ are the minimal subpresheaves of $C',C''$ satisfying $A'\cap\sd_3B'\neq\varnothing$, $A''\cap\sd_3B''\neq\varnothing$ and the vertical arrows are retractions.  
	There exists a unique dotted cubical function making the entire diagram commute.  
\end{lem}
\begin{proof}
  It suffices to take $C'$ and $C''$ atomic by minimality.  
  It suffices to take the cases $C'$ and $C''$ representable by naturality.  
  It further suffices to take the cases $C'$ and $C''$ isomorphic to $\BOX[1]$ because $\sd$ commutes with tensor products.  
  In that case, the lemma follows from inspection.
\end{proof}

\begin{prop}
  \label{prop:local.lifts}
  Suppose $\BOX$ is one of the following categories:
  $$\PARSIMONIOUSBOX,\DEL_1[\tau]^*,\DEL_1[\gamma_{\mins},\tau]^*,\DEL_1[\gamma_{\pls},\tau]^*,\DEL_1[\gamma_{\mins},\gamma_{\pls},\tau]^*.$$  
  There exist natural number $n_{(C,S)}$ and dotted cubical functions in
  \begin{equation*}
    \begin{tikzcd}
	        S\ar[r,dotted]\ar[d,hookrightarrow] & \BOX\boxobj{n_{(C,S)}}\ar[d,dotted]\\
	\sd_9C\ar{r}[below]{\epsilon^2_C} & C
	\end{tikzcd}
  \end{equation*}
   natural in objects $S\ira\sd_9C$ in the full subcategory of $\CUBICALSETS/\sd_9$ whose objects are inclusions $S\ira\sd_9C$ of non-empty subpresheaves $S$ of closed stars of vertices in cubical sets $C$, making the diagram commute.
\end{prop}

A proof for the exact same lemma, specialized for the case $\BOX=\DEL_1[\tau]^*$, is already given elsewhere \cite{krishnan2023cubicalII}.  
We include the following proof for completeness. 

\begin{proof}
  The cubical set $\epsilon_{\sd_3C}(S)$ is a non-empty subpresheaf of an atomic subpresheaf of $C$ [Lemma \ref{lem:collapse.star}]. 
  Let $A_{(C,S)}$ denote a choice of minimal atomic subpresheaf of $C$ containing $\epsilon_{\sd_3C}(S)$.
  There exists a unique minimal atomic subpresheaf $C_{S}\subset C$ with $\epsilon_{\sd_3C}(S)_{(C,S)}\cap\sd_3C_{S}\neq\varnothing$.  
  Note that $C_S$ is the unique minimal subpresheaf of $C$ with $A_{(C,S)}\cap\sd_3C_{S}\neq\varnothing$ and $B_{(C,S)}=A_{(C,S)}\cap\sd_3C_{S}$ is uniquely determined by $\epsilon_{\sd_3C}(S)$ by minimality.
  There exists a retraction $\pi_{(C,S)}$ making the right square in
  \begin{equation}
	  \label{eqn:bamfl.factorization}
      \begin{tikzcd}
	      S\ar[r,dotted]\ar[d,hookrightarrow] & A_{(C,S)}\ar[dotted]{r}[above]{\pi_{(C,S)}}\ar[d,hookrightarrow] & B_{(C,S)}\ar{d}[right]{\epsilon_C(B_{(C,S)}\ira\sd_3C)}\\
	      \sd_9C\ar{r}[below]{\epsilon_{\sd_3C}} & \sd_3C\ar{r}[below]{\epsilon_C} & C
	\end{tikzcd}
  \end{equation}
  commute [Lemma \ref{lem:collapse.star}].  
  There exists a dotted cubical function, unique by the middle arrow monic, making the left square commute by definition of $A_{(C,S)}$.  
  The cubical set $B_{(C,S)}$ is isomorphic to a representable [Lemma \ref{lem:collapse.star}]. 
  The left square above is natural in $(C,S)$.
  It therefore suffices to show that the right square above is natural in $(C,S)$.
  To that end, consider the $\STARS_9$-morphism defined by the left of the diagrams
  \begin{equation*}
    \begin{tikzcd}
	        S'\ar[dd,hookrightarrow]\ar{r}[above]{\alpha} & S''\ar[dd,hookrightarrow] &
	          A_{(C',S')}
		  \ar[d,hookrightarrow]
		  \ar{rrr}[above]{\support_{\sd_3}(\alpha,\beta)}
		  \ar{dr}[description]{\pi_{(C',A')}}
		& 
		& 
		& A_{(C'',S'')}
		  \ar[d,hookrightarrow]
		  \ar{dl}[description]{\pi_{(C'',A'')}}
		\\
		& &
		  \sd_3C'
		  \ar{d}[left]{\epsilon_{C'}}
		& B_{(C',S')}
		  \ar[r,dotted]
		  \ar{dl}[description]{\epsilon_{(C',S')}}
		& B_{(C'',S'')}
		  \ar{dr}[description]{\epsilon_{(C'',S'')}}
		& \sd_3C''
		  \ar{d}[right]{\epsilon_{C''}}
		\\
		\sd_9C'\ar{r}[below]{\sd_9\beta} & \sd_9C'' &
		  C'\ar{rrr}[below]{\beta}
		& 
		& 
		& C''
	\end{tikzcd}
  \end{equation*}

  Let $\epsilon_{(C,S)}$ denote $\epsilon_C(B_{(C,S)}\ira\sd_3C)$.
  Consider the right diagram.  
  There exists a unique dotted cubical function making the upper trapezoid commute [Lemma \ref{lem:retractions}].  
  The triangles commute by (\ref{eqn:bamfl.factorization}) commutative.  
  The lower trapezoid commutes because the outer rectangle commutes and the cubical functions of the form $\pi_{(C,A)}$ are epi.
  The desired naturality of the right square in (\ref{eqn:bamfl.factorization}) follows.
\end{proof}

\subsubsection{Nonpositive metric curvature}\label{sec:cat0.cubical.complexes}
The CAT(0) condition, a metric non-positive curvature condition on metric spaces, can be combinatorially characterized on cubical complexes equipped with their $\ell_2$-metrics.  
The cubical CAT(0) condition is useful in homotopy theory, where it is used to prove an equivariant Kan-Thurston Theorem \cite[Theorem 8.3]{leary2013metric} and give a sufficient local criterion for fundamental categories to embed into fundamental groupoids \cite{goubault2020directed}. 
Cubical subdivisions of representables are special cases of more general finite CAT(0) cubical complexes from finite distributive lattices and even more general finite \textit{distributive meet-semilattices}. 
We briefly sketch a construction of all finite CAT(0) cubical complexes as special cases of our nerve-like construction $\BOX[P]$, essentially due to observations made elsewhere [\cite[Theorem 2.5]{ardila2012geodesics} and \cite[Propositions 5.7, 5.8]{gonzalez2021finite}].

Take a meet-semilattice $L$ to be \textit{distributive} (c.f. \cite{gonzalez2021finite}) if given the data of \ldots
\begin{enumerate}
  \item \ldots a positive integer $n$ \ldots
  \item \ldots and $x_1,x_2,\ldots,x_n\in L$ whose supremum $x_1\vee_Lx_2\vee_L\cdots\vee_Lx_n$  exists in $L$
  \item \ldots and $y\in L$
\end{enumerate}
the supremum $(x_1\wedge_Ly)\vee_L(x_2\wedge_Ly)\vee_L\cdots\vee_L(x_n\wedge_Ly)$ exists in $L$ and 
$$(x_1\wedge_Ly)\vee_L(x_2\wedge_Ly)\vee_L\cdots\vee_L(x_n\wedge_Ly)=(x_1\vee_Lx_2\vee_L\cdots\vee_Lx_n)\wedge_Ly.$$

\begin{eg}
  A lattice is distributive if and only if it is distributive as a meet-semilattice.
\end{eg}

Suppose $\BOX$ lies in a chain of subcategories
$$\DEL_1[\tau]^*\subset\BOX\subset\POSETS$$
with $\DEL_1[\tau]^*$ wide in $\BOX$ and $\BOX$ EZ.  
Cubical sets of the form $\BOX[L]$ for all finite distributive meet-semilattices $L$ define all abstract finite CAT(0) cubical complexes when the posets are all finite distributive meet-semilattices, by a combination of observations made elsewhere [\cite[Theorem 2.5]{ardila2012geodesics} and \cite[Propositions 5.7, 5.8]{gonzalez2021finite}].
In fact, the edge-orientations on $\BOX[L]$ correspond exactly to a geodesic reachability relation: $x\leqslant_Ly$ if and only if every there is a geodesic on the $1$-skeleton of $|\BOX[L]|$, equipped with the geodesic metric defined by Euclidean distance on topological edges, from $\min\,L$ to $y$ that passes through $x$ \cite{ardila2012geodesics}.  
In particular for a finite distributive meet-semilattice $L$, the \textit{$\ell_2$-realization} of $\PARSIMONIOUSBOX[L]$ is a uniquely geodesic metric space.  

\subsection{Comparisons}
Let $\tri_{\shape{1}}$ denote the \textit{triangulation} functor
$$\tri_{\shape{1}}=(\snerve(\shape{1}\ira\CATS))_!:\hat{\shape{1}}\ra\hat\DEL.$$
for each small subcategory $\shape{1}$ of $\CATS$.
Triangulation is compatible with subdivision.  

\begin{lem}
  \label{lem:tri}
  Suppose $\BOX$ is one of the following variants:
$$\PARSIMONIOUSBOX,\BOX_1[\tau]^*,\BOX_1[\tau,\gamma_{\mins}]^*,\BOX_1[\tau,\gamma_{\pls}]^*,\BOX_1[\tau,\gamma_{\pls},\gamma_{\mins}]^*,\DEL_1[\tau,\gamma_{\mins},\gamma_{\pls},\diagonal]^*$$
  For each $k$, the following commutes up to natural monoidal isomorphism.
	\begin{equation*}
          \begin{tikzcd}
		  \CUBICALSETS\ar{r}[above]{\sd_{k+1}}\ar{d}[left]{\tri_\BOX}
 & \CUBICALSETS\ar{d}[right]{\tri_\BOX} \\
		  \SIMPLICIALSETS\ar{r}[below]{\sd_{k+1}} & \SIMPLICIALSETS
	  \end{tikzcd}
	\end{equation*}
   Along this natural isomorphism, we can make the identification $\tri_\BOX\eta_C=\eta_{\tri_\BOX C}$ natural in cubical sets $C$.
\end{lem}
\begin{proof}
  To verify the last line, it suffices to take $C=\BOX[1]$ by all functors and natural transformations Cartesian monoidal; in that case the last line follows from inspection. 
  It therefore suffices to construct an isomorphism
  $$\sd_{k+1}\nerve_\DEL\boxobj{n}=\sd_{k+1}\tri_\BOX\BOX\boxobj{n}\cong\tri_\BOX\sd_{k+1}\BOX\boxobj{n}=\tri_\BOX\BOX\left[(\boxobj{n})^{[k+1]}\right]$$
  natural in $\BOX$-objects $\boxobj{n}$ by the cocontintuity of $\tri_\BOX$ and both of the endofunctors labelled $\sd_{k+1}$.

  There exists the desired isomorphism natural in $\DEL_1$-objects $[n]=\boxobj{n}$ by inspection.
  
  The functors $\sd_{k+1}:\SIMPLICIALSETS\ra\SIMPLICIALSETS$, $(-)^{[k+1]}:\POSETS\ra\POSETS$ and $\nerve_\DEL$ are right adjoints and in particular preserve finite products.  
  And the construction $\BOX[-]$ sends $\POSETS$-products to tensor products.  
  Thus there exists the desired monoidal isomorphism natural in $\DEL_1^*$-objects $\boxobj{n}$.

  Both constructions $\sd_{k+1}\nerve_\DEL\boxobj{n}$ and $\tri_\BOX\BOX\left[(\boxobj{n})^{[k+1]}\right]$ are natural in $\BOX$-objects $\boxobj{n}$.  
  And all simplicial functions of the form $\sd_{k+1}\nerve_\DEL\phi$ and $\tri_\BOX\BOX\left[\phi^{[k+1]}\right]$ are uniquely determined by where they send vertices because non-degenerate simplices in $\sd_{k+1}\nerve_\DEL\boxobj{n}$ and $\tri_\BOX\BOX\left[(\boxobj{n})^{[k+1]}\right]$ are uniquely determined by their sets of vertices. 
  Therefore the natural identifications
  $$(\sd_{k+1}\nerve_\DEL\boxobj{n})_0=\DEL([0]^{\oplus(k+1)},\boxobj{n})=
  \POSETS([k+1],\boxobj{n})=\left(\tri_\BOX\BOX\left[(\boxobj{n})^{[k+1]}\right]\right)_0.$$
  of vertices extends to the desired natural isomorphism.

\end{proof}

For reasonable variants of $\BOX$ excluding the diagonals and reversals, the monad associated to the adjunction with left adjoint $\tri_{\BOX}$ is cocontinuous.

\begin{lem}
  \label{lem:qt.cocontinuous}
  Suppose $\BOX$ lies in a chain of submonoidal subcategories
  $$\DEL_1^*\subset\BOX\subset\PARSIMONIOUSBOX$$
  with $\DEL_1^*$ wide in $\BOX$ and $\BOX$ EZ.  
  The monad of the adjunction with left adjoint $\tri_{\BOX}$ is cocontinuous.
\end{lem}

A proof for the special case $\BOX=\DEL_1^*$ straigthforwardly adapts.  
For completeness, we spell out the proof in the more general context of this lemma. 

\begin{proof}
  Let $C$ denote a cubical set. 
  For brevity, let $\catfont{t},\catfont{q}$ denote the functors
  $$\catfont{t}=\tri_{\BOX},\quad\catfont{t}\dashv\catfont{q}.$$ 
  
  Consider the commutative diagram
  \begin{equation*}
    \begin{tikzcd}
	    \catfont{t}\BOX\boxobj{m_\theta}\ar{rrrrrr}[above]{\theta}\ar[dotted]{drrr}[description]{\snerve\,\phi} & & & {} & & & \catfont{t}C
	    \\
	    {}
	    & & {} & \catfont{t}\BOX\boxobj{n_\theta}\ar[dotted]{urrr}[description]{\tri\,\mu(\theta)}
	    \ar[l,phantom,"I"]
	    \ar[u,phantom,"II"]
	    \ar[r,phantom,"III"]
	    \ar[d,phantom,"IV"]
	    &
	    {}
	    \\
	    \DEL[i]\ar[dotted]{urrr}[description]{\snerve\phi(\theta)}\ar{uu}[left]{\snerve\iota} & & & {} & & & \DEL[1]\ar{llllll}[below]{\snerve(i\mapsto ni)}\ar[uu]\ar{ulll}[description]{\snerve(i\mapsto(i,i,\ldots,i))}
    \end{tikzcd}
  \end{equation*}
  in $\hat\DEL$, where $\iota$ denotes a monotone injection $[i]\ra\boxobj{m_\theta}$.
	There exist monotone extrema-preserving function $\phi(\theta)$ and cubical function $\mu(\theta)$ making I, III and hence also  IV commute because $\iota$ corestricts to an extrema-preserving montone function to an interval of $\boxobj{m_\theta}$.    
  We can take $\mu(\theta)$ to be a terminal such choice in $(\BOX/C)$, unique up to isomorphism, by $\BOX$ EZ.  
  Then $\mu(\theta)$ is the terminal choice of cubical function making III commute because the bottom diagonal arrows are all simplicial nerves of extrema-preserving monotone functions and monos in $\BOX$ are determined up to isomorphism by where they send their extrema [\LEMBoxInclusions].   
  Thus $\mu(\theta)$ does not depend on $\iota$.  
  It therefore follows that there exists a dotted monotone function $\phi$ making II commute.  
  Therefore the monic canonical cubical function 
  $$((\catfont{q}\catfont{t}\,\BOX[-])_!)(C)\ra\catfont{q}\catfont{t}C$$
  is epi and hence an isomorphism.
  Hence $\catfont{q}\catfont{t}$ is naturally isomorphic to the cocontinuous functor $(\catfont{q}\catfont{t}\BOX[-])_!$.
\end{proof}

\section{Homotopy}\label{sec:homotopy}

Cubical sets are interpretable as abstract combinatorial descriptions of geometric objects.  
Each such interpretation implicitly defines a homotopy theory on cubical sets based on a homotopy theory of some geometric objects.  

\subsection{Classical}
Let $|-|$ denote topological realization 
$$|-|:\SIMPLICIALSETS\ra\TOP.$$

There exists a unique homeomorphism $\dihomeo_{S;k+1}:|\sd_{k+1}S|\cong|S|$, natural and Cartesian monoidal in simplicial sets $S$, that is linear on each geometric simplex and hence characterizable by the followingn rule (c.f. \cite{ehlers2008ordinal}):
$$\dihomeo_{\DEL[n];k+1}(v)=\nicefrac{v}{k+1}\in|\DEL[n]|,\,v\in(\sd_{k+1}\DEL)_0=\DEL[k+1]_0.$$

More generally let $|-|$ denote the dotted functor making the diagram
\begin{equation*}
  \begin{tikzcd}
    \hat{\shape{1}}\ar{d}[left]{\tri_{\shape{1}}}\ar[dotted]{r}[above]{|-|} & \TOP\ar[d,equals]\\
    \SIMPLICIALSETS\ar{r}[below]{|-|} & \TOP .
  \end{tikzcd}
\end{equation*}
commute up to natural homeomorphism, for each subcategory $\shape{1}\subset\CATS$, 
Call $|C|$ the \textit{topological realization} of the $\hat{\shape{1}}$-object $C$.

\subsubsection{Homotopies}
We construct some homotopies needed to prove the main results.

\begin{lem}
  \label{lem:sd.approx}
  There exists a homotopy 
  $$|\epsilon_S|\sim|\dihomeo_{S;3}|:|\sd_{k+1}S|\ra|S|$$
  natural in simplicial sets $S$.
\end{lem}
\begin{proof}
  It suffices to take the case $S$ representable by naturality.  
  In the case $S=\DEL[n]$, linear interpolation defines the desired homotopies natural in $\DEL$-objects $[n]$ because $|\DEL[-]|$ sends $\DEL$-morphisms to convex linear maps between topological simplices.  
\end{proof}

\begin{lem}
  \label{lem:unit.equivalence}
  Suppose $\BOX$ lies in a chain of categories
  $$\DEL_1^*\subset\BOX\subset\POSETS$$
  with $\DEL_1^*$ wide in $\BOX$. 
  Write $\qua_\BOX$ for the right adjoint to $\tri_\BOX$.  
  The map $|\eta_C|$ is a homotopy equivalence, where $\eta_C:C\ra(\qua_\BOX\tri_\BOX[-])_!(C)$ is the cubical function induced by the unit of $\tri_\BOX\dashv\qua_\BOX$.
\end{lem}
\begin{proof}
  Note that $\eta_C$ is the cubical function
  $$\eta_C:C\ira\colim_{\BOX\boxobj{n}\ra C}\nerve_\BOX\boxobj{n}$$
  induced by inclusion $\BOX\ira\CATS$.  
  Then $|\eta_C|=|\tri_\BOX\eta_C|$ induces an equivalence of fundamental groupoids by an application of simplicial approximation.
  Moreover, $\tri_\BOX\eta_C$ admits a retraction $\rho_C$ natural in $C$ by the zig-zag identities for adjunctions.  
  Therefore $\tri_\BOX\eta_C$ and $\rho_C$ induce mutually inverse integral simplicial chain homotopy equivalences, natural in cubical sets $C$, by an application of the Acyclic Models Theorem.   
  Thus $|\eta_C|$ induces a homology isomorphism for all choices of local coefficients.  
  Therefore $|\eta_C|$ is a homotopy equivalence of topological spaces by the Whitehead Theorem.
\end{proof}

\subsubsection{Model structures}
There are often two types of model structures on presheaf categories, both often Quillen equivalent to the usual model structure on topological spaces.  
Recall that a \textit{test model structure} is a model structure on $\hat{\shape{1}}$ in which the weak equivalences are those $\hat{\shape{1}}$-morphisms $\psi:A\ra B$ for which $|\nerve_\DEL(\shape{1}/\psi)|:|\nerve_\DEL(\shape{1}/A)|\ra|\nerve_\DEL(\shape{1}/B)|$ is a homotopy equivalence and cofibrations are the monos, for each small category $\shape{1}$.   
Call a \textit{classical model structure} a model structure on $\hat{\shape{1}}$ left induced along topological realization $|-|:\hat{\shape{1}}\ra\TOP$ for a small subcategory $\shape{1}$ of $\CATS$.

\begin{eg}
  On $\SIMPLICIALSETS$, the test and classical model structures coincide.
\end{eg}

\begin{eg}
  Suppose $\BOX$ is one of the following variants:
  $$\DEL_1^*,\DEL_1[\gamma_{\mins}]^*,\DEL_1[\gamma_{\pls}]^*,\DEL_1[\tau,\gamma_{\mins},\gamma_{\pls},\diagonal]^*.$$
  Then the test and classical model structures on $\hat\BOX$ exist and coincide and are equivalent to the usual model structure on $\TOP$ along the Quillen equivalence whose left map is topological realization [\cite[Theorem 8.4.38]{cisinskiprefaisceaux}, \cite[Proposition 2.3, Theorem 2.4]{cavallotype}, \cite[Theorem 6.3.6]{awodey2024equivariant}].  
  The test and classical model structures do not generally coincide for symmetric monoidal variants of $\BOX$ that exclude the diagonals (e.g. \cite[p5. \P3]{awodey2024equivariant}) like $\PARSIMONIOUSBOX$.
\end{eg}

\begin{prop}
  \label{prop:classical.model.structure}
  Suppose $\BOX$ is one of the following variants:
  $$\BOX=\DEL_1[\tau]^*,\DEL_1[\tau,\gamma_{\mins}]^*,\DEL_1[\tau,\gamma_{\pls}]^*,\DEL_1[\tau,\gamma_{\mins},\gamma_{\pls}]^*,\PARSIMONIOUSBOX.$$
  There exists a model structure on $\hat\PARSIMONIOUSBOX$ in which \ldots
  \begin{enumerate}
    \item \ldots a weak equivalence $\psi$ is characterized by $|\psi|$ a homotopy equivalence
    \item \ldots the cofibrations are the monos
  \end{enumerate}
  Topological realization defines the left map of a Quillen equivalence from $\hat\PARSIMONIOUSBOX$ equipped with this model structure to the category of topological spaces equipped with its classical model structure.
  If $\BOX=\DEL_1[\tau,\gamma_{\mins}]^*,\DEL_1[\tau,\gamma_{\pls}]^*,\DEL_1[\tau,\gamma_{\mins},\gamma_{\pls}]^*,\PARSIMONIOUSBOX$, then localization by weak equivalences in this model structure preserves finite products.
\end{prop}

The proof uses the following observation from the literature.  
A category $\mathscr{A}$ admits a model structure left transferred along a left adjoint $L:\mathscr{A}\ra\mathscr{B}$ if $\mathscr{A}$ and $\mathscr{B}$ are both locally presentable, the model structure on $\mathscr{B}$ is additionally cofibrantly generated, and for each $\mathscr{A}$-object $o$, the $\mathscr{A}$-morphism $\id_o\amalg\id_o:o\amalg o\ra o$ factors into a composite $\rho_o\iota_o$ with $L\iota_o$ a cofibration in $\mathscr{B}$ and $L\rho_o$ a weak equivalence in $\mathscr{B}$ \cite[special case of Theorem 2.2.1]{hess2017necessary}.  

\begin{proof}
  For brevity, define $\catfont{t}=\tri_\BOX$, $\catfont{t}\dashv\catfont{q}$, and let $\eta,\epsilon$ denote the unit and counit of $\catfont{t}\dashv\catfont{q}$.
  
  There exists a model structure on $\CUBICALSETS$ left induced by the classical model structure on $\SIMPLICIALSETS$ along triangulation $\tri_{\BOX}$ because $\CUBICALSETS$ and $\SIMPLICIALSETS$ are locally presentable, the classical model structure on $\SIMPLICIALSETS$ is cofibrantly generated, and for each cubical set $C$ the fold morphism $\id_C\amalg\id_C:C\amalg C\ra C$ factors as a composite of $\BOX[\delta_{\mins}]\amalg\BOX[\delta_{\pls}]$ followed by $C\otimes\BOX[\sigma]$ and $\tri_\BOX(\BOX[\delta_{\mins}]\amalg\BOX[\delta_{\pls}])$ is monic and hence a cofibration in $\SIMPLICIALSETS$ and $\tri_\BOX\BOX[\sigma]=\DEL[\sigma]$ is a weak equivalence in $\SIMPLICIALSETS$ [\cite[Theorem 2.2.1]{hess2017necessary}].
  This model structure has the desired weak equivalences and cofibrations because the test model structure on $\SIMPLICIALSETS$ is left induced by the classical model structure on $\TOP$ left induced by topological realization $|-|$.  
  Moreover, $|-|$ is the left map of a Quillen adjunction by construction.  

  The functor $\catfont{t}$ is the left map of a Quillen adjunction and hence induces a left adjoint $h\catfont{t}:h\CUBICALSETS\ra h\SIMPLICIALSETS$ between homotopy categories of associated classical model structures.  
  The functor $h\catfont{t}$ is essentially surjective because $\catfont{t}(\DEL_1^*\ira\BOX)_!=\tri_{\DEL_1^*}$ is the left map of a Quillen equivalence.
  The functor $h\catfont{t}$ is therefore a categorical equivalence because $\eta$ is object-wise a weak equivalence [Lemma \ref{lem:unit.equivalence}].
  Thus the left Quillen map $|-|$ is the left map of a Quillen equivalence.
\end{proof}

An intrinsically cubical description of the fibrant objects and weak equivalences in the classical model structure \cite[Theorem 3.20]{krishnan2022uniform} straightforwardly adapts for the case $\BOX=\PARSIMONIOUSBOX$.
Moreover, $|-|:\hat\PARSIMONIOUSBOX\ra\TOP$ sends test fibrations to maps that are Serre fibrations up to a natural self homotopy equivalence, in the following sense.  

\begin{lem}
  \label{lem:dold.fibration}
  Suppose $\BOX$ is one of the following variants:
  $$\BOX=\DEL_1[\tau]^*,\DEL_1[\gamma_{\mins}]^*,\DEL_1[\tau,\gamma_{\mins}]^*,\DEL_1[\gamma_{\pls}]^*,\DEL_1[\tau,\gamma_{\pls}]^*,\PARSIMONIOUSBOX.$$
  Consider a fibration $\psi$ in the classical model structure on $\hat\BOX$ and commutative diagram
  \begin{equation}
    \label{eqn:dold}
    \begin{tikzcd}
	    {|C|}
	    \ar[dr]\ar{rrrr}[above]{f}
	    \ar{dd}[left]{|C|\times(\{0\}\ira\I)} 
	    & & & & {|E|}\ar{dd}[right]{|\psi|}
	    \ar{dlll}[description]{|\epsilon_E|\dihomeo_{E;3}^{-1}}
      \\
      & {|E|}\ar{dd}[right,,near start]{|\psi|}
      \\
      {|C|\times\I}\ar[ur,dotted]\ar{rrrr}[below]{g}\ar[dr] & & & & {|B|}\ar{dlll}[description]{|\epsilon_B|\dihomeo_{B;3}^{-1}}
      \\
      & {|B|}
    \end{tikzcd}
  \end{equation}
  There exists a dotted map making the entire diagram commute in $\TOP$.  
\end{lem}
\begin{proof}
  Let $A$ denote an object in the category of all atomic subpresheaves of $C$ and inclusions between them.  
  For brevity, let $d_C=|\epsilon_C|\dihomeo^{-1}_{C;3}:|C|\ra|C|$.
  Define $M\psi$, $\iota_{\psi}$, $\rho_{\psi}$, and $\kappa_{\psi}$ by the commutativity of the following diagram, natural in cubical functions $\psi:C'\ra C''$, in which the outer square is co-Cartesian:
	\begin{equation*}
          \begin{tikzcd}
		  C'\otimes\BOX[1]\ar{dr}[description]{\psi(C'\otimes\BOX[\sigma])}\ar{rr}[above]{\iota_{\psi}} & & M\psi\ar{dl}[description]{\rho_{\iota}}
		  \\
		  & C''
		  \\
		  C'\ar{rr}[below]{\psi}\ar{uu}[left]{C'\otimes\BOX[\delta_{\mins}]} & & C''\ar{uu}[right]{\kappa_{\psi}}\ar{ul}[description]{\id_{C''}}
	  \end{tikzcd}
	\end{equation*}
	
  It suffices to take $C$ finite; the general case follows because $|C|$ is the homotopy colimit over all spaces $|B|$, where $B$ is taken over all finite subpresheaves of $C$.  
  Thus for $k\gg 0$, $f\dihomeo^k_C,g\dihomeo^k_C$ both map closed cells into open stars of vertices by $|C|$ compact.  
  Therefore it suffices to take $f,g$ to map each closed cell in $|C|$ and $|C\otimes\sd^k\BOX[1]|$, respectively, to an open star of a vertex.  
  It therefore further suffices to take $g$ to map each closed cell in $|C\otimes\BOX[1]|$ to an open star of a vertex by induction.  

  There exist terminal choice of $(\BOX/E)$-object $\alpha_A:\BOX\boxobj{e_A}\ra E$, choice of $(\BOX/B)$-object $\beta_A:\BOX\boxobj{b_A}\ra B$ and lifts $\lambda_{f;A}$ and $\lambda_{g;A}$ of $d_Ef|A\ira C|$ against $|\alpha_A|$ and $d_Bg|(A\ira C)\otimes\BOX[1]|$ against $|\beta_A|$, natural in $A$ [Proposition \ref{prop:local.lifts}].  
  There exists a $\BOX$-morphism $\phi_A:\boxobj{e_A}\ra\boxobj{b_A}$ natural in $A$ such that $\psi\alpha_A=\beta_A\BOX[\phi_A]$ by terminality in our choice of $\alpha_A$.   
  Note that the $\BOX[\phi_A]$'s and hence also the $\iota_{\BOX[\phi_A]}$'s induce simplicial functions between weakly contractible triangulations and therefore define acyclic cofibrations.  
  Define $\nu_\psi$ and $F\psi$ so that $\nu_\psi:E\ra F\psi$ is the pushout of $\amalg_A\iota_{\BOX[A]}$ along $\amalg_A\alpha_A:\amalg_A\BOX\boxobj{e_A}\ra E$.
  In particular note that $\nu_\psi$ is an acyclic cofibration.
  Define the solid diagram
  \begin{equation*}
    \begin{tikzcd}
	    {|C|}\ar[r]
	    \ar{d}[left]{|C|\times(\{0\}\ira\I)} 
	    & {|\colim_A\BOX\boxobj{e_A}|}\ar{rrr}[above]{|\colim_A\BOX\boxobj{e_A}\ra E|}\ar{d}[description]{|\nu_\psi|}
	    & & & {|E|}\ar{d}[right]{|\psi|}
      \\
	      {|C|\times\I}\ar[r]
	    & {|F\psi|}\ar{rrr}[below]{|\colim_AM\BOX[\phi_A]\ra B|}\ar[urrr,dotted]
	    & & & {|B|}
    \end{tikzcd}
  \end{equation*}
  as follows. 
  The top left horizontal arrow is induced by the $\lambda_{f;A}$'s.  
  The bottom left horizontal arrow is induced by the $|\kappa_{\BOX[\phi_A]}|\lambda_{g;A}$'s
  The top right horizontal arrow is induced by the $\alpha_A$'s.
  The bottom right horizontal arrow is induced by the $\beta_A\rho_{\BOX[\phi_A]}$'s.  
  Then the composite of the top horizontal arrows is $d_Ef$ and the composite of the bottom horizontal arrows is $d_Bg$.  
  The left square commutes up to homotopy because the left square would commute if the middle vertical arrow were replaced by $|\colim_A\BOX[\phi_A]|$.    
  There exists a dotted map making the right square commute because in the classical model structure, $\colim_A\iota_{\BOX[\phi_A]}$, an iterated pushout of acyclic cofibrations $\iota_{\BOX[\phi_A]}$, is an acyclic cofibration, and $\psi$ is a fibration.   
  Therefore the composite map $\lambda:|C|\times\I\ra|E|$ in the above diagram therefore lifts $d_Bg$ against $|\psi|$ and satisfies $\lambda(|C|\times(\{0\}\ira\I))\sim d_Ef$.  
  Then $\lambda$ can be replaced by a dotted map making (\ref{eqn:dold}) commute by $|C|\times(\{0\}\ira\I)$ a Hurewicz cofibration.  
\end{proof}

The classical model structure on $\hat\PARSIMONIOUSBOX$ is therefore \textit{type-theoretic} in a suitable sense.  

\begin{thm}
  \label{thm:proper.test}
  Suppose $\BOX$ is one of the following variants:
  $$\BOX=\DEL_1^*,\DEL_1[\tau]^*,\DEL_1[\gamma_{\mins}]^*,\DEL_1[\tau,\gamma_{\mins}]^*,\DEL_1[\gamma_{\pls}]^*,\DEL_1[\tau,\gamma_{\pls}]^*,\PARSIMONIOUSBOX.$$
  The classical model structure on $\hat\BOX$ is proper.
\end{thm}

A relatively recent criterion for a presheaf model category to be right proper \cite{gambino2017frobenius} does not apply, at least obviously, to the classical model category $\hat\PARSIMONIOUSBOX$.  
The proof of the theorem instead mimics a classical argument for simplicial sets: that geometric realization reflects weak equivalences and sends fibrations and weak equivalences to the fibrations and weak equivalences of a right proper model category.

\begin{proof}
  For brevity, let $d_C=|\epsilon_C|\dihomeo^{-1}_{C;3}:|C|\ra|C|$.
  
  Consider the left Cartesian square $\hat\BOX$ among the diagrams
  \begin{equation*}
    \begin{tikzcd}
	    E\times_BA\ar[dr,phantom,"\lrcorner",near start]\ar{r}[above]{\psi^*\iota}\ar[d,two heads] & E\ar[two heads]{d}[right]{\psi}\\
      A\ar{r}[below]{\iota} & B
    \end{tikzcd}\quad
    \begin{tikzcd}
	    {|E\times_BA|}\ar[r,hookrightarrow]\ar[d,two heads]\ar[dr,phantom,"\lrcorner",near start] & {|E|}\ar{d}[right]{|\psi|}\\
	    {|A|}\ar[hookrightarrow]{r}[below]{|\iota|} & {|B|}
    \end{tikzcd}
  \end{equation*}
  where $\psi$ and $\iota$ are, respectively, a fibration and a weak equivalence.
  It suffices to show that the top horizontal arrow in the left diagram is a weak equivalence.  
  The case $\iota$ an acyclic fibration follows from the model categorical axioms.  
  It therefore suffices to consider the case where $\iota$ is an acyclic cofibration, and in particular an inclusion of presheaves, because every weak equivalence in a model structure factors into an acyclic cofibration followed by an acyclic fibration.  
  In that case, application of $|-|$ yields the right Cartesian square in $\TOP$.  

  There exist retraction $r$ to $|\iota|$ and homotopy $h:|\iota|r\sim\id_{|B|}$ relative $|A|$ because $|\iota|$ is an acyclic cofibration of topological spaces [Proposition \ref{prop:classical.model.structure}].  
  Thus $d_Bh|\psi|:d_B|\iota|r|\psi|\sim d_B|\psi|=|\psi|d_E$.  
  Hence there exists a map $\tilde{s}:|E|\ra|E|$ and homotopy $\tilde{h}:\tilde{s}\sim d_E$ such that $|\psi|\tilde{h}=d_Bh$ [Lemma \ref{lem:dold.fibration}]; in particular $|\psi|\tilde{s}=d_B|\iota|r|\psi|$ .
  The map $\tilde{s}:|E|\ra|E|$ has image in $|E\times_BA|$ because $d_B$ restricts and corestricts to $d_A$ by naturality and hence the image of $|\psi|\tilde{s}=d_B|\iota|r|\psi|$ lies in $|A|$.  
  Thus $\tilde{s}:|E|\ra|E|$ admits a corestriction $\tilde{r}$ to $|E\times_BA|$ because topological realization sends inclusions to topological embeddings.
  
  Note $\tilde{h}$ defines a homotopy from $|E\times_BA\ira E|\tilde{r}=\tilde{s}$ to $d_E$. 
  Note that $d_B$ and $d_{E}$ respectively restrict and corestrict to $d_A$ and $d_{E\times_BA}$ by naturality.  
  The restriction of $\tilde{h}$ to $|E\times_BA|\times\I$ has image $|E\times_BA|$ because the restriction of $|\psi|\tilde{h}=d_Bh$ to $|A|\times\I$, defined by a map of the form $d_A(|A|\times\I\ra|A|)$, has image $|A|$.  
  Thus $\tilde{h}$ restricts and corestricts to a homotopy $\tilde{r}|E\times_BA\ira E|\sim d_{E\times_BA}$ because topological realization sends inclusions to topological embeddings.
  
	Thus $\tilde{r}|E\times_BA\ira E|\sim d_E\sim\id_{|E|}$ and $|E\times_BA\ira E|\tilde{r}\sim d_{E\times_BA}\sim\id_{|E\times_BA|}$ [Lemmas \ref{lem:tri} and \ref{lem:sd.approx}].  
  Thus $|E\times_BA\ira E|:|E\times_BA|\ira|E|$ is a homotopy equivalence and hence $\psi^*\iota:E\times_BA\ira E$ is a weak equivalence.
\end{proof}

\subsection{Exotic}
Topological realizations often naturally admit extra geometric information encoded in the combinatorics of the presheaves.  
Let $\mathscr{M}$ be a category equipped with a cocontinuous functor $\mathscr{M}\ra\SETS$.  

\begin{eg}
  Motivating examples of $\mathscr{M}$ are categories of \ldots
  \begin{enumerate}
	  \item pseudometric spaces and $1$-Lipschitz maps
	  \item uniform spaces and uniform maps
	  \item directed topological spaces and directed maps. 
  \end{enumerate}
\end{eg}

Fix a small subcategory $\shape{1}$ of $\CATS$. 
Sometimes $|\shape{1}[o]|$ naturally lifts to an $\mathscr{M}$-object natural in $\shape{1}$-objects $o$.  

\begin{eg}
  \label{eg:geometric.cubes}
  Suppose $\BOX$ lies in a chain of subcategories
  $$\DEL_1^*\subset\BOX\subset\POSETS$$
  with $\DEL_1^*$ wide in $\BOX$. 
  The topological $n$-cubes $\I^n=|\BOX\boxobj{n}|$ naturally admit $\ell_\infty$ metrics, unique uniformities compatible with their topologies, and partial orders during them into (pseudo)metric spaces, uniform spaces, and directed topological spaces.
  If $\BOX$ additionally does not have diagonals, then the topological $n$-cubes $\I^n=|\BOX\boxobj{n}|$ alternative admit $\ell_p$ metrics for all $1\leqslant p\leqslant\infty$ turning them into (pseudo)metric spaces.  
\end{eg}

Suppose $|\shape{1}[-]|$ lifts to a solid horizontal functor in the diagram
\begin{equation*}
  \begin{tikzcd}
	  \shape{1}\ar[rr]\ar{d}[left]{\shape{1}[-]} & & \mathscr{M}\ar[d]\\
	  \hat{\shape{1}}\ar[dotted]{urr}[description]{|-|_{\catfont{GEO}}}\ar{r}[below]{|-|} & \TOP\ar[r] & \SETS.
  \end{tikzcd}
\end{equation*}

We can define geometric realization $|-|_{\catfont{GEO}}$ as the left Kan extension of the top horizontal functor along the Yoneda embedding $\shape{1}[-]$, necessarily making the entire diagram commute where $\TOP\ra\SETS$ denotes the forgetful functor and $\mathscr{M}\ra\SETS$ is the given cocontinuous functor.

\begin{eg}
  \label{eg:geometric.realizations}
  Suppose $\BOX$ lies in a chain of subcategories
  $$\DEL_1^*\subset\BOX\subset\POSETS$$
  with $\DEL_1^*$ wide in $\BOX$.  
  Then \textit{directed, $\ell_\infty$, and uniform realizations} can be defined on cubical sets.
  If $\BOX$ does not contain the diagonal $[1]\ra[1]^2$, then \textit{$\ell_p$ realizations} can also be defined on cubical sets for all $1\leqslant p\leqslant\infty$.
\end{eg}

Homotopy theories of cubical sets defined in terms of such geometric realization functors have been explored elsewhere.  
The adaptation of cubical approximation theorems for $\BOX=\PARSIMONIOUSBOX$ with respect to these more exotic homotopy theories are mostly straightforward and not given in this paper.  


\bibliography{gv}{}
\bibliographystyle{plain}

\end{document}